\documentclass[11pt]{amsart}
\usepackage[dvips]{graphicx}
\usepackage{amsmath}
\usepackage{amsthm}

\usepackage{amssymb}
\usepackage{mathrsfs}
\usepackage{latexsym}
\theoremstyle{plain}
\newtheorem{thm}{Theorem}[section]
\newtheorem{lem}[thm]{Lemma}
\newtheorem{cor}[thm]{Corollary}
\newtheorem{prop}[thm]{Proposition}
\theoremstyle{definition}
\newtheorem{rem}[thm]{Remark}
\newtheorem{defn}[thm]{Definition}

\numberwithin{equation}{section} 

\newcommand{\R}{\mathbb{R}}
\newcommand{\C}{\mathbb{C}}

\newcommand{\N}{\mathbb{N}}

\newcommand{\Lp}{\widehat{L}^p}
\newcommand{\Lpp}{\widehat{L}^{p'}}

\newcommand{\A}{\alpha}

\newcommand{\CS}{\mathfrak{S}}

\newcounter{kotaeflg}
\newcommand{\kotae}[1]{
\ifodd \arabic{kotaeflg}
#1
\fi
}

\begin{document}
\title[ ]%
      {Well-posedness and decay estimates for
      1D nonlinear Schr\"odinger equations with Cauchy data in $L^p$}
\author[R. Hyakuna]{Ryosuke Hyakuna}
\address[R. Hyakuna]{Waseda University}
\email{107r107r@gmail.com}
\begin{abstract}
An $L^p$-theory of local and global solutions to the one dimensional nonlinear Schr\"odinger equations with power type nonlinearities $|u|^{\A-1}u,\,\A>1$ is developed.  Firstly, some twisted local well-posedness results in subcritical $L^p$-spaces are established for $p<2$.  This extends YI.Zhou's earlier results for the gauge-invariant cubic NLS equation.  Secondly, by a similar functional framework, the global well-posedness for small data in critical $L^p$-spaces is proved, and as an immediate consequence, $L^{p'}$-$L^p$ type decay estimates for the global solutions are derived, which are well known for the global solutions to the corresponding linear Schr\"odinger equation.  Finally, global well-posedness results for gauge-invariant equations with large $L^p$-data are proved, which improve earlier existence results, and from which it is shown that the global solution $u$ has a smoothing effect in terms of spatial integrability at any large time.  Linear weighted estimates and bi-linear estimates for Duhamel type operators in $L^p$-spaces play a central role in proving the main results.
\end{abstract}
\maketitle
%
%

\section{Introduction}
This paper is concerned with the Cauchy problem for the one dimensional nonlinear Schr\"odinger equation (NLS)
\begin{equation}
iu_t +u_{xx}+N(u) =0,\quad u|_{t=0} = \phi, \label{NLS}
\end{equation}
where the nonlinearity $N(u)$ is such that
\begin{eqnarray}
&&N(0)=0 \label{power1} \\ 
&&|N(u)-N(v)|\le C(|u|^{\A-1}+|v|^{\A-1} )|u-v|, \label{power2}
\end{eqnarray}
for some $ \A>1$.
Typical nonlinearities of such type are power type (not necessarilly gauge invariant) ones:
\begin{equation*}
|u|^{\A-1}u,\quad |u|^{\A-1}\bar{u},\quad |u|^{\A},etc\cdots.
\end{equation*}
As is well known, extensive investigations have been made on the initial value problem (\ref{NLS}) with (\ref{power1})-(\ref{power2}).
Usually, in the study of nonlinear dispersive equations including NLS, initial data are assumed to be in a function space whose norm is characterized by some kind of square integrability.  Examples of such data spaces are
 $L^2$-space, $L^2$-based Sobolev spaces $H^s$, weighted $L^2$-spaces and so on, and at present a lot of results on (\ref{NLS}) are known in the framework of those function  spaces.  Among them are the most basic local and global
existence theorem for (\ref{NLS}) in the $L^2$-space.  We summarize them as a proposition below:\\
\textbf{Proposition A.}

\begin{enumerate}
\item
\textit{If $1<\A<5$, {\rm (\ref{NLS})} with {\rm (\ref{power1})--(\ref{power2})} is locally well posed for large data in $L^2$.
\item
If $\A=5$, {\rm (\ref{NLS})} with {\rm(\ref{power1})--(\ref{power2})} is globally well posed for small data in $L^2$.
\item
If $1<\A<5$, {\rm (\ref{NLS})} with $N(u)=|u|^{\A-1}u$ is globally well posed for large data in $L^2$.}

\end{enumerate}

On the other hand, when data $\phi$ are not characterized by any kind of square integrability, much less is known about the solvability of (\ref{NLS}).  Most popular examples of
such data spaces are $L^p$-spaces ($p\neq 2$).  Unfortunately, as early as the 1960s H\"ormander  \cite{Hormander} showed that the linear Schr\"odinger equation corresponding to (\ref{NLS}) is
ill posed if $p\neq 2$ (see also \cite{Brenner}).  Especially, the corresponding solution to the linear equation for $L^p$-data does not belong to $L^p$ unless $p=2$.  
This implies that one cannot expect the usual well-posedness of NLS in $L^p,\,p\neq 2$ with persistence property of solutions.  For that reason, it seems analysis of nonlinear Schr\"odinger equations in the $L^p$-setting has been out of interest of many researchers.  However, as mentioned below $L^p$-theory of solutions has various interesting aspects, especially from the viewpoint of Fourier analysis.  The aim of this paper is three-fold but stated in one sentence: develop a satisfactory mathematical theory of solutions to (\ref{NLS}) in $L^p$ spaces as 
 a natural extension of standard $L^2$-theory.  For this purpose, throughout the paper, the initial data $\phi$ are assumed to be in the mere $L^p$-space and we do not impose any extra conditions on them except the smallness of their $L^p$-norm in the second main result.  To examine the solvability of (\ref{NLS}) with (\ref{power1})-(\ref{power2}), the standard scaling argument is useful.  For a given $\A$, the space $L^p$ is said to be subcritical if
$p>\frac{\A-1}{2}$, critical if $p=\frac{\A-1}{2}$ and supercritical if $p<\frac{\A-1}{2}$.  In view of the standard local well-posedness theory in the $H^s$-framework, one may expect the local existence for data in subcritical $L^p$-spaces and
the global existence for small data in critical $L^p$-spaces.  Our first two subject are concerning a generalization to the $L^p$-setting of Proposition A (i), (ii), which are local
existence in the subcritical space and global existence for small data in the critical space in the $L^2$-framework.  Now let us briefly explain our motivation for studying (\ref{NLS}) in $L^p$-spaces for $p<2$.
It comes from the corresponding linear Schr\"odinger equation in $\R^d$
\begin{equation}
iu_t+\Delta u=0,\quad u|_{t=0}=\phi.  \label{Linear}
\end{equation}
It is well known that, if $\phi \in L^p,\,1\le p\le 2$, the solution $u$ to the linear Schr\"odinger equation
is in $C(\R\setminus \{0 \} ;L^{p'})$ and has the $L^{p'}$-$L^p$ type decay estimate
\begin{equation}
\|u(t,\cdot )\|_{L^{p'}} \le (4\pi |t|)^{-d(\frac{1}{p}-\frac{1}{2}) } \label{Ldecay}
\|\phi \|_{L^p}
\end{equation}
if $\phi $ are in $L^p,\,1\le p\le 2$.  The estimate is very comparable with the Hausdorff--Young inequality
\begin{equation}
\|\hat{f} \|_{L^{p'}} \le C \|f\|_{L^p}
\end{equation}
for $1\le p \le 2$ and one may think (\ref{Ldecay}) reflects the fact that the operator $U(t):\phi \mapsto u(t)$ 
acts like a kind of the Fourier transform parametrized by $t\in \R\setminus \{0\}$.

The first motivation comes from a natural question of whether or not
a similar decay estimate holds for solutions to the nonlinear Schr\"odinger equations (NLS).  The question may be rephrased in the context of the Fourier analysis: does the nonlinear operator $\mathscr{N}(t):$data\,$\mapsto$ solution ($t\neq 0$) corresponding to NLS also act as a parametrized Fourier transform and a Haudorff--Young like estimate hold for $\mathscr{N}(t)$?  The first aim of the paper is clear.  We want to obtain similar, $L^p$-$L^{p'}$ type
decay estimates
\begin{equation}
\|u(t,\cdot)\|_{L^{p'}} \le C(\|\phi\|_{L^p}) \times |t|^{-(\frac{1}{p}-\frac{1}{2})},\quad t\neq 0 \label{intronldecay}
\end{equation}
 for global solutions to (\ref{NLS}) with data in
$L^p$-spaces.  However, as far as the author knows, there are no previous works which focus on data in mere $L^p$-spaces.   In fact, at present, even
the existence of global solutions such that $u(t) \in L^{p'}$ is unclear let alone their decay properties.  A hint on the study in this direction is from YI. Zhou's pioneering work on the local well-posed ness for the gauge-invariant 1D cubic NLS (i.e, $N(u)=|u|^2u$).  In \cite{Zhou}, he showed that for any $\phi \in L^p,\,1<p<2$, there exists a local solution $u$ to (\ref{NLS}) such that
\begin{equation}
U(-t)u(t) \in C([-T,T] ;L^p) \label{ipersistency}
\end{equation}
where $U(t)$ denotes free the propagator and $T>0$ is a life span of the local solution.
His work gives some clues for obtaining the desired decay estimate.  Indeed, if we find a global solution to (\ref{NLS}) with $\phi \in L^p$ such that
(\ref{ipersistency}) holds for arbitrarily large $T>0$ and such that
\begin{equation}
M\triangleq \sup_{t \in \R} \|U(-t)u(t) \|_{L^p} <\infty, \label{lpbb}
\end{equation}
we then immediately derive the an $L^{p'}$-$L^p$ type decay estimate by combining (\ref{lpbb}) with (\ref{Ldecay}).  In view of the scaling argument above, it is not unnatural to expect the existence of such global solutions for small data in critical $L^p$-spaces, that is $\phi \in L^{\frac{\A-1}{2}}$.
Although Zhou's work gives some light on the local well-posedness theory for NLS in $L^p$-spaces, one would still face difficulties when trying to get the decay estimate for small solutions above.  Firstly, the proof of the local results in \cite{Zhou} relies on a subtle cancellation property of nonlinearities and multi-linear interpolation between a key $L^1$ estimate and a classical $L^2$ estimate, and therefore, it is difficult to apply Zhou's method to general power type nonlinearities except for $N(u)=|u|^{2m}u,\,m\in \N$.  Secondly, the functional framework employed in \cite{Zhou} cannot work well for the critical nonlinearities.  One can observe this kind of difficulty more easily for the Hartree equation (see \cite{107H}).  The second aim of the paper is to provide a new local theory of solutions with the twisted property (\ref{ipersistency}) via a more popular functional framework in which one may cover more general power type nonlinearities (including non-gauge invariant and non-integer power type ones) and Cauchy data in critical $L^p$-spaces.  In this paper we employ the standard Lebesgue spaces  $L^q([0,T]; L^r)$ with $2/q+1/r=1/p$ as auxiliary spaces.  As is well known these are very commonly used function spaces to solve (\ref{NLS}).  For example, if $p=2$, it is well known that a local $L^2$-solution is constructed in
the space $L^{\infty}([0,T]; L^2)\cap L^q([0,T] ; L^r)$ for suitable choices of $q,r$ such that $2/q+1/r=1/2$ via the contraction mapping principle.  This is a classical result by Y. Tsutsumi \cite{yT}.  So our local result can be regarded as a natural generalization of the standard local $L^2$-theory of solutions in one space dimension.

Finally, the third aim is concerned with the global well-posedness for (\ref{NLS}) with
$N(u)=|u|^{\A-1}u$ without any smallness assumptions. When $p=2$, a global solution can easily be obtained thanks to the $L^2$-conservation and
the life span estimate $T\ge \|\phi \|_{L^2}^{-\frac{4(\A-1)  }{5-\A  }}$ upon getting a local solution on $[0,T]$ via the fixed point argument.  In the mere $L^p$-setting with $p\neq 2$, on the other hand, no conservation laws are available directly and therefore, it is nontrivial to
extend the local solution to a global one.  Nonetheless, in \cite{VV}, Vargas and Vega
studied the 1D cubic case and by means of a splitting argument, they showed that if $\phi \notin L^2$ but $\phi$ is sufficiently
close to an $L^2$-function in some sense, then a global solution can be established.  
Later in \cite{107H2}, \cite{107}, \cite{107T2} we considered the case of more general power $\alpha $ and use their approach to construct a global solution for $L^p$-Cauchy data
if $p<2$ is sufficiently close to $2$.  This is based on the natural idea that any $L^p$-function is ``close'' to an $L^2$-function if $p$ is near $2$.  
  Thus, global existence results can be established below the $L^2$-conservation law.  However, an interesting question on global well-posedenss in $L^p$ remains unsolved.  In earlier works, the global solutions are constructed  in Lebesgue spaces $L^q_{loc}(L^r)$.  For example, the global solution constructed by Vargas and Vega \cite{VV} belongs to $L^3_{loc}(L^6)$-space.  Now our local well-posedness results say the global solution satisfies
the twisted property (\ref{ipersistency}) at least for some finite $T>0$.  Therefore, one may ask: does the global solutions for $L^p$-data established in the above literature satisfy
\begin{equation*}
U(-t)u(t) \in C([0,T] ; L^p)
\end{equation*}
for any large $T>0$?  This is our third aim.  We want to prove global well-posedness in $L^p$ for large $L^p$-data.  Needless to say, this is a natural generalization of Proposition A (iii) into the
$L^p$-framework.  In this paper, we answer the question positively, if $\A>3$ and $p$ is sufficiently close to $2$.

\bigskip
\noindent\textit{List of notation.}
\begin{itemize}
\item
$a'$ means the conjugate exponent of $a$: \,$1/a+1/a'=1$.
\item
Let $0<q,r\le \infty$ and let $I\subset \R$.  The space-time Lebesgue space 
$L^q(I ; L^r)=L^q_I(L^r)$ is characterized by its norm
\begin{equation*}
\|u\|_{L_I^q(L^r)}\triangleq \left(\int_I \|u (t,\cdot)\|_{L^r}^q dt \right)^{\frac{1}{q}}
\end{equation*}
with the usual modification when $q=\infty$.  In particular, we write
$L_{\R}^q(L^r)=L^q(L^r)$.
\item
The Fourier transform $\mathcal{F}$ and the inverse Fourier transform $\mathcal{F}^{-1}$ are defined by
\begin{equation*}
\mathcal{F}f\triangleq \hat{f}\triangleq
\int_{\R} e^{-i\xi x} f(x) dx,\quad \mathcal{F}^{-1} f\triangleq
\int_{\R} e^{ix \xi} f(x) dx.
\end{equation*}
\item
$U(t)$ is the free propagator i.e. the solution to the linear Schr\"odinger equation (\ref{NLS}) is denoted by $U(t)\phi$.  Therefore, (\ref{Ldecay}) is expressed as the linear estimate
\begin{equation}
\|U(t)\phi \|_{L^{p'}} \le (4\pi |t|)^{-d(\frac{1}{p}-\frac{1}{2})} \|\phi \|_{L^p}.
\label{decay}
\end{equation}
We also write $(U^{-1}\phi)(t) \triangleq U(-t)\phi$.
\item
The hat $L^p$ spacee $\Lp$ is defined by
\begin{equation*}
\{\phi \in \mathcal{S}' \,| \, \mathcal{F}\phi \in L^{p'} \}.
\end{equation*}
\item
Throughout the paper, $c,C,C_1,C_2,\cdots,$ are positive constants which may not be the same at each appearance.

\end{itemize}

Before we state our main results, we introduce spaces of functions with the property (\ref{ipersistency}).  Let $E$ be a Banach space of functions on $\R$ and let $I\subset \R$ be an interval.  We define the space $C_{\mathfrak{S}}(I ;E)$ by
\begin{equation*}
C_{\mathfrak{S}}(I ; E) \triangleq \{ u:I\times \R \to \C\,|\,\,
(t\mapsto U(-t)u(t) )\in C(I ; E) \,\}.
\end{equation*}
Note that if a function space $E$ is such that the unitarity property
\begin{equation*}
\|U(t)\phi \|_{E}=\|\phi \|_{E},\,\,\forall t\in I,\,\,\phi \in E
\end{equation*}
holds
then we have $C(I ;E)=C_{\mathfrak{S}}(I ;E)$.
We define
\begin{equation}
\mathfrak{L}_I^{\infty}(E)\triangleq \mathfrak{L}^{\infty}(I;E)\triangleq \{u:I\times \R\to \C\,\,|\,
U(-t)u(t) \in L^{\infty}(I \,;\, E) \,\} \label{infty}
\end{equation}
equipped with the norm
\begin{equation*}
\|u\|_{\mathfrak{L}^{\infty}_I (E)}\triangleq \|U^{-1}u\|_{L^{\infty}_I (E)}.
\end{equation*}

We define the exponent $r\triangleq r(\A)$ by
\begin{equation}
r\triangleq
\begin{cases}
\A+1 &\text{if}\,\, 1<\A \le 3\\
2(\A-1) & \text{if}\,\, 3\le \A <5.
\end{cases}
\label{defnr}
\end{equation}
We first present the local well-posedness results for $L^p$-data, which extends Zhou's earlier results \cite{Zhou} to more general, not necessarily gauge invariant, power-like nonlinearities:
\begin{thm} (Large data local well-posedness in subcritical $L^p$-spaces)\label{LWP}\\
Assume that $1<\A<5$ and
\begin{equation*}
\max \left(\frac{\A-1}{2} ,\frac{2(\A-1)}{\A},\frac{\A+1}{\A} \right) <p\le 2.
\end{equation*}
Let $r$ be given by (\ref{defnr}) and let $q$ be such that
\begin{equation}
\frac{2}{q}+\frac{1}{r}=\frac{1}{p}. \label{thmscaling}
\end{equation}
Then for any $\phi \in L^p(\R)$ there exists $T\sim\|\phi \|_{L^p}^{-\frac{2p(\A-1)  }{2p-\A+1 } }>0$ and
a unique local solution $u$ to (\ref{NLS}) such that
\begin{equation*}
u \in C_{\CS }([0,T] ; L^p(\R) ) \cap L^q ([0,T ] ; L^r(\R)).
\end{equation*}
Moreover, the map $\phi \mapsto u(t)$ is locally Lipschitz from $L^p$ to $C_{\CS}([0,T] ; L^p)$.  Furthermore,
the solution $u$ has a smoothing effect in terms of spatial integrability:
\begin{equation}
u|_{]0,T] \times \R} \in C(]0,T] ;L^{p'}(\R)). \label{lsmoothing}
\end{equation}
\end{thm}

\bigskip
The above functional framework works well in the case of
critical nonlinearities and yields the wanted decay estimates
for global solutions for small $L^p$-data along with their existence:

\begin{thm} (Small data global well-posedness in critical $L^p$-spaces with $L^{p'}$-$L^p$ decay estimates)\label{SGWP}\\
Assume that
\begin{equation*}
4<\A<5.
\end{equation*}
 Let 
\begin{equation*}
p=\frac{\A-1}{2}
\end{equation*}
and let $q$ be determined by
\begin{equation*}
\frac{2}{q}+\frac{1}{2(\A-1)}=\frac{1}{p}.
\end{equation*}
Then there is a sufficiently small positive $\epsilon>0$ such that:
for any $\phi \in L^p$ with $\|\phi \|_{L^p}<\varepsilon$ there is 
a unique global solution $u$ to (\ref{NLS}) such that
\begin{equation*}
u \in C_{\CS}(\R ; L^p (\R)) \cap L^q(\R ; L^{2(\A-1)} (\R))
\end{equation*}
and
\begin{equation}
M\triangleq \sup_{t\in \R} \|U(-t)u(t) \|_{L^p} <\infty. \label{mbounded}
\end{equation}
Moreover, $u$ has a smoothing effect in terms of spatial integrability
\begin{equation*}
u|_{(\R\setminus \{0 \}) \times \R} \in C(\R\setminus\{ 0\} \,;\, L^{p'}(\R) ).
\end{equation*}
Furthermore, as an immediate consequence of {\rm(\ref{decay})} and {\rm (\ref{mbounded})},
$L^{p'}$-$L^p$ type decay estimates for the solution are derived:
\begin{equation*}
\|u(t) \|_{L^{p'}}\le (4\pi |t|)^{ -(\frac{1}
{p}-\frac{1}{2}) } \times M,\quad \forall t\neq 0.
\end{equation*}
\end{thm}

\bigskip
Finally, we give global well-posedness results for large $L^p$-data in the case of
gauge invariant nonlinearity:
\begin{thm}\label{LGWP} (Large data global well-posedness for the gauge-invariant equation)
Assume that $N(u)=|u|^{\A-1}u$ with $3 \le \A <5$ and
\begin{equation*}
\max \left( \frac{(\A-1)(\A+3)}{2\A^2+2\A-4}, \frac{(\A-1)(3\A+5)}{2(\A^2-2\A+5)} \right)<p\le 2.
\end{equation*}
Let $r=2(\A-1)$ and let $q$ be determined by (\ref{thmscaling}).
Then the local solution to (\ref{NLS}) for $\phi \in L^p$ given by Theorem \ref{LWP}
extends to a unique global solution such that
\begin{equation*}
u\in C_{\CS} (\R ; L^p (\R)) \cap L^q_{loc}(\R ;L^r(\R) ).
\end{equation*}
Moreover, the global solution has a smoothing effect in terms of spatial integrability at any large time $t\neq 0$:
\begin{equation}
u|_{(\R\setminus \{ 0\} ) \times \R} \in C(\R\setminus \{ 0\} \,; \, L^{p'}(\R)\,).
\end{equation}
\end{thm}

\begin{rem}
One natural generalization of the classical, $L^2$-based Sobolev space $H^s$ to the $L^p$-setting is the Bessel potential space $H^s_p$ defined by
\begin{equation*}
H^s_p(\R^n) \triangleq \{ \varphi \in \mathcal{S}'(\R)\,| \, \mathcal{F}^{-1}[ (1+|\xi|^2)^{\frac{1}{2}} \widehat{\varphi}] \in L^p(\R^n)\,\}.
\end{equation*}
So one may wonder if similar local results hold for $\phi \in H^s_p(\R)$ for suitable $s,p,\A$.  Perhaps, this is possible if we adopt the Lebesgue-Besov spaces $L^q(I : B^s_{r,\theta})$ as
 auxiliary space and
argue as in Cazenave-Weissler \cite{CW}. This is another advantage of the Strichartz technique.  In fact, it is not easy to obtain
similar well-posedness results via Zhou's functional framework in \cite{Zhou}.  Since
the space $H^s_p$ is not obtained from complex interpolation between $L^1$ and $H^s(=H^s_2)$ (see \cite{BL}), 
one cannot get key estimates for $H^s_p$ by just replacing $L^2$ estimates with $H^s$ estimates
in the argument in \cite{Zhou}.
\end{rem}

\bigskip

\noindent\textbf{Other related results}.\quad We here present some other results related to the main theorems above.

By the standard argument of  Cauchy sequence, we easily deduce the existence of the scattering state for the small solution in Theorem \ref{SGWP}:
\begin{cor}\label{scatter}
Assume that $4<\A<5$ and 
\begin{equation*}
p=\frac{\A-1}{2}.
\end{equation*}
Let $u$ be the global solution for small $\phi \in L^p$ given by Theorem {\rm\ref{SGWP}}.  Then 
there exist $\phi_{\pm}\in L^p$ such that
\begin{equation}
\lim_{t\to\pm \infty}\|U(-t)u(t) -\phi_{\pm} \|_{L^p}=0.
\end{equation}

\end{cor}

\medskip
Finally, we present a global well-posedness result for (\ref{NLS}) with $N(u)=|u|^{\A-1}u$ in $\Lp$-spaces.  Since the $\Lp$-spaces enjoy the unitarity property
\begin{equation*}
\|U(t)\phi \|_{\Lp}=\|\phi\|_{\Lp},\quad \forall t\in\R,
\end{equation*}
local well-posedness in $\Lp$ with the usual persistence property is expected.  Several results on the local existence of (\ref{NLS}) for $\phi \in \Lp$
have been reported.  By an argument similar to the proof of Theorem \ref{LGWP}, we can prove the existence of global solutions with persistence property
for any large time.

\begin{thm}(Large data global well-posedness in $\Lp$) \label{hatglobal}
Assume that $\A>2$ and the nonlinearity is of type $N(u)=|u|^{\A-1}u$ and
\begin{equation*}
2\le p <\min \left(\A+1, \frac{3\A+5}{2\A}\right).
\end{equation*}
 Let $q$ be determined by the relation
\begin{equation*}
\frac{2}{q}+\frac{1}{\A+1}=\frac{1}{p}.
\end{equation*}
Then for any $\phi \in \Lp$ there exists a unique global solution $u$ such that
\begin{equation*}
u\in C(\R ;\Lp(\R))\cap L^q(\R ; L^{\A+1} (\R)).
\end{equation*}

\end{thm}

\medskip
\noindent\textbf{Strategy}.
As mentioned above, we consider (\ref{NLS}) for data in the mere $L^p$-space and no other
additional assumption on data.  Therefore, the Cauchy problem (\ref{NLS}) should be discussed in the pure $L^p$-framework.  We seek the solution of the corresponding integral equation 
\begin{equation}
u(t)=U(t)\phi +\int^t_0 U(t-s) N(u(s)) ds. \label{introint}
\end{equation}
In the setting of $L^2$ and $H^s$, the so-called Strichartz inequalities are standard tools to estimate (\ref{introint}).  Recall that $(q,r)$ is admissible if $r\ge 2$ and
\begin{equation*}
\frac{2}{q}+\frac{1}{r}=\frac{1}{2}.
\end{equation*}
The well-known Strichartz estimate in the case of one space dimension are of the form
\begin{equation}
\|U(t) \phi \|_{L^{\beta}(\R;L^{\kappa}(\R))} \le C\|\phi \|_{L^2(\R)} \label{L2Str}
\end{equation}
and
\begin{equation}
\left\| \int^t_0 U(t-s) F(s) ds \right\|_{L^{\beta}(\R; L^{\kappa}(\R))} \le C\|F\|_{L^{\sigma}(\R; L^{\rho}(\R))}. \label{L2Strh},
\end{equation}
which hold true if $(\beta, \kappa)$ and $(\sigma',\rho')$ are admissible.  Local well-posedness of (\ref{NLS}) in $L^2$-space can be proved by the contraction mapping principle
after estimating the right hand side of (\ref{introint}) using (\ref{L2Str})--(\ref{L2Strh}).  So when one tries to solve (\ref{NLS}) for $\phi \in L^p,\,p\neq 2$, it is quite natural to wonder if estimates similar to (\ref{L2Str}) and (\ref{L2Strh}) hold in 
the $L^p$-setting.  Actually the estimate
\begin{equation}
\|U(t)\phi \|_{L^q(L^r)} \le C\|\phi \|_{L^p} \label{introStrp}
\end{equation}
is true for $p<2$ and some $q,r$ satisfying $2/q+1/r=1/p$.  However, these estimates are not sufficient in order to achieve our goal.  The main difficulty is the property (\ref{ipersistency}) cannot be obtained only from (\ref{introStrp}) --(\ref{L2Strh}).

The key idea of the paper is to exploit another generalized Strichartz type estimate
\begin{equation}
\|U(t)\phi \|_{L^q(L^r)} \le C\|\phi \|_{\Lp} \label{introFS}
\end{equation}
which is stronger than (\ref{introStrp}).   We do not need this inequality to estimate the linear part of the integral equation but by duality, it is equivalent to an $L^p$-estimate for Duhamel type operator, which then leads to a key bilinear type estimate.  Consequently, the billnear Duhamel estimate derived essentially from (\ref{introFS})
yields the basic nonlinear estimate
\begin{equation}
\sup_{t\in [0,T]} \left\|\int^t_0 U(-s)N(u(s))ds \right\|_{L^p} \le C\|u\|_{L^q_{[0,T]}(L^r)}^{\A-1} \|U^{-1}u\|_{L^{\gamma}_{[0,T]}(L^p)}. \label{introbilinear}
\end{equation}
which is valid for $p>4/3$ and some suitable choices of $q,r,\gamma$.  In particular,
taking $\gamma=\infty$ is allowed and thus, combining (\ref{introbilinear}) with (\ref{introStrp}) and (\ref{L2Strh})
, we expect that the fixed point of the integral equation could be found in the space
\begin{equation*}
\{u \,|\, U^{-1}u \in L^{\infty}([0,T] ; L^p) \,\} \cap L^q([0,T] ;L^r).
\end{equation*}
 The bilinear estimate (\ref{introbilinear}) play a fundamental role throughout the paper not only
for the proof of the local well-posedness result.  Indeed, taking $\gamma<\infty$
provides a kind of a priori estimate and combined with earlier global existence 
results, we get a global solution $u$ such that 
\begin{equation*}
U(-t)u(t)\in C(\R ;L^p)
\end{equation*}
even for large data $\phi\in L^p$ in the case of gauge-invariant nonlinearities.

\section{Key estimates}\label{SectionKey}
In this paper, we construct a solution to (\ref{NLS}) in space-time Lebesgue spacees $L^q(L^r)$.  To this end
we here gather various generalizations of the standard Strichartz estimates used in the usual $L^2$-framework.
\subsection{Strichartz type estimates}
We first present generalizations the estimate (\ref{L2Str}) for the solution to the homogeneous equation
\begin{equation*}
iu_t +u_{xx}=0,\qquad u|_{t=0}=\phi.
\end{equation*}
We first introduce some sets of pairs of exponents to describe these generalized estimates.

\begin{defn}
Let $1\le  p \le 2$.  \begin{enumerate}\item
$\mathscr{S}(p)$ is the set of all pairs of exponents $(q,r)$ such that
\begin{equation}
\frac{2}{\beta}+\frac{1}{\kappa}=\frac{1}{p} \label{scaling}
\end{equation}
and
\begin{equation*}
0<\frac{1}{\kappa}<\frac{1}{2},\quad 0<\frac{1}{\beta}<\frac{1}{2}-\frac{1}{\kappa}.
\end{equation*}
\item
$\widehat{\mathscr{S}}(p)$ is the set of all pairs of exponents $(\beta ,\kappa)$ such that
$(\beta ,\kappa) \in \mathscr{S}(p)$ and
\begin{equation*}
0<\frac{1}{\kappa}<\frac{1}{2},\quad 0<\frac{1}{\beta}<\min \left(\frac{1}{2}-\frac{1}{\kappa},\frac{1}{4} 
\right).
\end{equation*}
or $(\beta, \kappa)=(4,r),\,r>4$.
\end{enumerate}
\end{defn}

In this paper we exploit two types of generalization of the homogeneous Strichartz estimate (\ref{L2Str}).  The first one
is the estimate of the form (\ref{L2Str}) with $\|\phi\|_{L^2}$ in the right hand side being replaced by $\|\phi\|_{L^p},\,p\le 2$, which seems
a natural extension of (\ref{L2Str}) into the $L^p$-framework.

\begin{prop}
Let $1<p\le 2$.  For any $(\beta,\kappa) \in \mathscr{S}(p)$ the estimate
\begin{equation}
\|U(t) \phi \|_{L^{\beta}(L^{\kappa})} \le C\|\phi \|_{L^p} \label{StrLp}
\end{equation}
holds true.
\end{prop}

\begin{proof}\,

See \cite[Theorem 2.1]{Kato}.

\end{proof}

In view of the Plancherel identity $\|\phi \|_{L^2}=c\|\widehat{\phi}\|_{L^2}$, there can be another $L^p$ generalization of the Strichartz estimate: 
replace $\|\widehat{\phi}\|_{L^2}$ with $\|\widehat{\phi}\|_{L^{p'}}$.  In the case of one space dimension, this kind of generalization holds true
as long as $p> 4/3$:

\begin{prop}\label{GFSprop}
Let $4/3 <p \le 2$.  For any $(\beta,\kappa) \in \widehat{\mathscr{S}}(p)$ the estimate
\begin{equation}
\|U(t) \phi \|_{L^{\beta}(L^{\kappa})} \le C\| \phi \|_{\Lp}. \label{GFS}
\end{equation}
holds true.

\end{prop}

\begin{proof}
See e.g. \cite[Lemma 4]{107T}.

\end{proof}

Next we present a generalization of the estimate (\ref{L2Strh}) for the solution of inhomogeneous equations
\begin{equation*}
iu_t +u_{xx}=F,\quad u|_{t=0}=0.
\end{equation*}

\begin{prop} \label{inhomo}
Assume that
\begin{equation*}
2+\frac{2}{\beta}+\frac{1}{\kappa}=\frac{2}{\sigma}+\frac{1}{\rho},
\end{equation*}
\begin{equation*}
2<\kappa<\infty,\quad 1<\rho <2
\end{equation*}
and
\begin{equation*}
0<\frac{1}{\beta}<\frac{1}{2}-\frac{1}{\kappa},\quad \frac{3}{2}-\frac{1}{\rho}<\frac{1}{\sigma}<1.
\end{equation*}
Then the estimate
\begin{equation*}
\left\| \int^t_0 U(t-s)F(s) ds \right\|_{L^{\beta}(L^{\kappa})} \le C\|F\|_{L^{\sigma}(L^{\rho})}
\end{equation*}
is true.

\end{prop}

\subsection{$L^p$-Duhamel estimate}  Now we present our key $L^p$ estimate in this paper.  Recall the two generalized Strichartz estimate for the homogeneous equation above.
By the Hausdorff--Young inequality, we easily see that (\ref{GFS}) is stronger than (\ref{StrLp}) if $p<2$ while $\widehat{\mathscr{S}}(p) \subset \mathscr{S}(p)$.  It is sufficient to use (\ref{StrLp}) in order to estimate the linear contribution $U(t)\phi$ of the integral equation corresponding to
(\ref{NLS}).  However, it turns out that the stronger estimate (\ref{GFS})  is also very useful.  Indeed, it plays an important role in the estimate of the Duhamel part of the integral equation which
lead to the twisted persistency property of the solution (\ref{ipersistency}).
To see this, we rewrite Proposition \ref{GFSprop} by taking dual of (\ref{GFS}):


\begin{cor} \label{DFS}
Assume that $2\le p<4$.  Let $J\subset \R$ be an interval. Let $\mathfrak{I}(J)$ be the set of
all intervals included in $J$.  Then
\begin{equation}
\sup_{I \in \mathfrak{I}(J)} \left\| \int_I U(-s)f(s) ds \right\|_{\Lp}
\le \|f \|_{L_J^{\tilde{\sigma}} (L^{\tilde{\rho}})} \label{DFSest}
\end{equation}
for any $\tilde{\sigma},\tilde{\rho}$ such that $(\tilde{\sigma}', \tilde{\rho}') \in\widehat{\mathscr{S}}(p')$.

\end{cor}

Now, from (\ref{DFSest}), we derive our key bilinear $L^p$ estimate for a Duhamel type operator.

\begin{prop}\label{bilinear}
\begin{enumerate}
\item
Assume $4/3<p\le 2$.  Let $\sigma$ be such that
$(\sigma', \rho') \in \mathscr{S}(p)$ for an exponent $\rho$. Then 
\begin{equation}
\sup_{t \in [0,T]} \left\|
\int^t_0 U(-s)F(s) ds \right\|_{L^p} \le C \| s^{\frac{1}{p}-\frac{1}{2} }F(s) \|_{L_{[0,T]}^{\sigma}(L^{\rho})}\label{WDFS}
\end{equation}
for any $F$ such that $s^{\frac{1}{p}-\frac{1}{2}}F(s)\in L^{\sigma}_{[0,T]}(L^{\rho})$.
\item
Let $p,\sigma,\rho$ be as in (i).  Then
\begin{equation}
\sup_{t \in [0,T]} \left\|
\int^t_0 U(-s)F(s) ds \right\|_{L^p}
\le C \|f_1 \|_{L_{[0,T]}^{\sigma} (L^{\frac{\sigma}{2\sigma-2}   } )} \|U^{-1}f_2 \|_{L^{\infty}_{[0,T]}(L^p)} \label{BLWDFS}
\end{equation}
for any $f_1,f_2,F$ such that $(f_1, U^{-1}f_2 )\in 
L_{[0,T]}^{    \sigma  }  (L^{  \frac{\sigma}  {2\sigma-2} }    ) \times L^{\infty}_{[0,T]}(L^p)$
and $|F| \le |f_1| |f_2|$.
\end{enumerate}
\end{prop}

\begin{proof} To deduce the estimate, we need the following factorization formula for the evolution operator 
$U(-t)$(see \cite[Chapter 2]{Caz}):
\begin{equation*}
U(-t)=M_t^{-1} \mathcal{F}^{-1} D_t^{-1} M_t^{-1},\quad t\neq 0,
\end{equation*}
where 
\begin{equation*}
M_t : w\mapsto e^{i\frac{|x|^2}{4t}},\quad (D_t w)(x)\triangleq (4\pi i t)^{-\frac{1}{2}} w\left(\frac{x}{4\pi i t}\right).
\end{equation*}
We now start the proof of the estimate.  Fix $t \in [0,T]$ and set
\begin{equation*}
[Iu](t) \triangleq \int^t_0 U(-s) F(s) ds.
\end{equation*}
Using the factorization of $U(-t)$, we may write
\begin{eqnarray*}
[Iu](t) &=&\int^t_0 M_s^{-1} \mathcal{F}^{-1} D_s^{-1} M_s^{-1} F(s) ds \\
&=& \int^t_0 M_s^{-1} \mathcal{F}^{-1} g(s) ds,
\end{eqnarray*}
where
\begin{equation*}
g(s)\triangleq D_s^{-1} M_s^{-1} F(s).
\end{equation*}
We observe that $\mathcal{F}M_s \mathcal{F}^{-1}=cU(-1/s),\,s\neq 0$.  Thus
for any $t\in (0,T]$, we have
\begin{eqnarray*}
\mathcal{F}[\overline{Iu (t)}] &=& \int^t_0 (\mathcal{F} M_s \mathcal{F}^{-1}) R\overline{g(s)} ds \\
&=& \int^t_0 U(-1/s) R \overline{g(s)} ds \\
&=& \int^{\infty}_0 U(-\tau) \frac{R\overline{ g(1/\tau)}}{\tau^2} d\tau,
\end{eqnarray*}
where $(Rf)(x)\triangleq f(-x)$.
Therefore, by Corollary \ref{DFS}, we have
\begin{eqnarray*}
\| [Iu](t) \|_{L^p} &=& \left\| \mathcal{F}^{-1} [Iu](t) \right\|_{\Lpp} \\
&=& \left\| \overline{\mathcal{F}^{-1} [Iu](t) }\right\|_{\Lpp} \\
&=& \left\| \mathcal{F} [\overline{Iu}](t) \right\|_{\Lpp} \\
&=& \left\| \int^{\infty}_{t^{-1}} U(-\tau) \frac{R\overline{g(1/\tau)}}{\tau^{2}} d\tau \right\|_{\Lpp} \\
&\le & C\|\tau^{-2} g(1/\tau) \|_{L^{\sigma}_{[T^{-1},\infty)}(L^{\rho})} \\
&=& C\| s^{2-\frac{2}{\sigma}} g(s) \|_{L^{\sigma}_{[0,T]}(L^{\rho})}.
\end{eqnarray*}
Since $|g(s,x) |=4\pi |s|^{\frac{1}{2}} \left| F(s, 4\pi s x) \right|$, we have
\begin{equation*}
\|g(s,\cdot )\|_{L^{\rho}}=4\pi |s|^{\frac{1}{2}-\frac{1}{\rho}} \|F(s,\cdot)\|_{L^{\rho}}.
\end{equation*}
Thus we get
\begin{equation}
\| [Iu](t) \|_{L^p} \le C \| s^{\frac{1}{p}-\frac{1}{2} }F(s) \|_{L_{[0,T]}^{\sigma}(L^{\rho})} 
\end{equation}
for any $t\in [0,T]$.  In order to prove (ii), we examine the right hand side further.  Fix $s\in [0,T]$.
Applying the H\"older inequality, we have
\begin{equation*}
\|F(s) \|_{L^{\rho}} \le \|f_1(s) f_2(s) \|_{L^{\rho}}
\le \|f_1(s) \|_{L^{ \frac{\sigma}{2\sigma-2}  } } \| M_s^{-1} f_2(s) \|_{L^{p'}}
\end{equation*}
where note that $\frac{\sigma}{2\sigma-2} \in [1,\infty]$ since $\sigma \le 2$.  Now noting the relation
\begin{equation*}
\| D_s^{-1} f\|_{L^k} =c|s|^{\frac{1}{2}-\frac{1}{k}} \|f\|_{L^k}
\end{equation*}
for any $0<k<\infty$ and the Hausdorff--Young inequality and recalling the factorization formula above, we get
\begin{equation*}
\|F(s) \|_{L^{\rho}} \le s^{\frac{1}{2}-\frac{1}{p}}\|f_1(s) \|_{L^{ \frac{\sigma}{2\sigma-2}  } }  \|U(-s) f_2(s) \|_{L^p}.
\end{equation*}
Finally, taking $L^{\sigma}$-norm after multiplying $s^{1/p-1/2}$, we get 
the wanted estimate.
\end{proof}

\begin{rem}
 We observe that the key bilinear type, Duhamel $L^p$ estimate is obtained from the generalized Strichartz estimate (\ref{GFS}) via the change of the time variable $t\mapsto 1/t$.  This implies that the global-in-time version of (\ref{GFS}) should be
exploited even for obtaining our local results in $L^p$.

\end{rem}
\subsection{Hausdorff--Young like property}
Finally, in this section, we present an elementary lemma
which shows that the twisted persistence property (\ref{ipersistency})
implies smoothing effects in terms of spatial integrability.
\begin{lem}
Assume that $1\le p\le 2$.  Let $I\subset \R$ and $u\in C_{\CS}(I ;L^p(\R^d))$.  
Then $u |_{I\setminus \{ 0\} \times \R^d} \in C(I\setminus \{0 \} ; L^{p'}(\R^d) )$.

\end{lem}\label{H--Y}
\begin{proof}\quad The assertion is straightforward.  See \cite[Lemma 2.5]{107H2} for a precise proof.

\end{proof}

\section{Proof of Theorem \ref{LWP} and \ref{SGWP}} \label{SLWP}
We first prove Theorem \ref{LWP}.  Assume that
\begin{equation}
\max\left( \frac{\A-1}{2}, \frac{4}{3}, \frac{\A+1}{\A} \right) <p \le 2. \label{proofprange}
\end{equation}
We seek a solution of the corresponding integral equation
\begin{equation}
u(t)=U(t)\phi +i \int^t_0 U(t-s) N(u(s)) ds. \label{inteq}
\end{equation}
Let $T>0$.  We define
\begin{equation*}
X_T\triangleq \mathfrak{L}^{\infty} ([0,T] ; L^p(\R)) \cap L^q ([0,T] ; L^r(\R) )
\end{equation*}
endowed with the norm
\begin{equation*}
\| u\|_{X_T}\triangleq \max \left( \| U^{-1}u\|_{L^{\infty}_{[0,T]}(L^p)}, \, \|u\|_{L^q_{[0,T]}(L^r)} \right),
\end{equation*}
where $q,r$ will be determined below.
We set
\begin{equation*}
\mathscr{T}u\triangleq U(t)\phi +i \int^t_0 U(t-s) N(u) ds
\end{equation*}
and find a fixed point of $\mathscr{T}$ in the complete metric space
\begin{equation*}
\mathscr{V}(T,R)\triangleq \{ u\in X_T \,|\,\,\,
\|u\|_{X_T} \le R \}
\end{equation*}
equipped with the distance
\begin{equation*}
d(u_1,u_2) \triangleq \|u_1-u_2\|_{X_T}.
\end{equation*}

\noindent\textbf{Case 1}. \quad We assume $1<\A\le 3$.   Here, let $q,r$ be such that
\begin{equation*}
r=\A+1,\quad \frac{2}{q}+\frac{1}{r}=\frac{1}{p}.
\end{equation*}
Then it is easy to check that $(q,r)\in \mathscr{S}(p)$ if $p>(\A+1)/\A$.  Therefore, we have
\begin{eqnarray*}
\|U(t) \phi \|_{L^q_{[0,T]}(L^r)} &\le & C \|\phi \|_{L^p}.\\
\end{eqnarray*}
In order to estiamte $L^q(L^r)$ norm of the nonlinear part, we let $\gamma$ be determined by the relation
\begin{equation*}
2+\frac{1}{p}=\frac{2}{\gamma}+\frac{\A}{\A+1}.
\end{equation*}
Then the quadruple $(\beta,\kappa,\sigma,\rho)$ satisfies the assumption of Proposition \ref{inhomo} if $p>(\A+1)/\A$.  Therefore, we have
\begin{eqnarray*}
\left \|\int^t_0 U(t-s)N(u(s)) ds \right\|_{L^q_{[0,T]}(L^r)} &\le C &\left\| |u|^{\A-1} u \right\|_{L^{\gamma}_{[0,T]}(L^{\frac{\A+1}{\A}})} \\
&=&C \|u\|^{\A}_{L^{\A\gamma}_{[0,T]}(L^{\A+1})} \\
&\le & CT^{1-\frac{\A-1}{2p}} \|u\|_{L^q_{[0,T]}(L^r)},
\end{eqnarray*}
where we have used the H\"older's inequality in the last step and this is possible as long as $p>(\A-1)/2$.

Next we estimate $\mathfrak{L}^{\infty}_{[0,T]}(L^p)$-norm of the integral equation.  Clearly, $\|U(t)\phi \|_{\mathfrak{L}^{\infty}_{[0,T]}(L^p)}=\|\phi \|_{L^p}$.
Now we estimate $L^p$ norm of the Duhamel type term using (\ref{WDFS}).  We set
\begin{equation*}
\rho =\frac{\A+1}{\A},\quad \sigma =\left(1+\frac{1}{2(\A+1)}-\frac{1}{2p}\right)^{-1}.
\end{equation*}
Then $(\sigma', \rho') \in \widehat{\mathscr{S}}(p)$ if $p>(\A+1)/\A$ and $1<\A\le 3$.  Thus, by (\ref{WDFS}), we get
\begin{equation*}
\left\| \int^t_0 U(t-s) N(u(s))ds \right\|_{\mathfrak{L}^{\infty}_{[0,T]} (L^p)} \le \| s^{\frac{1}{p}-\frac{1}{2}} N(u)\|_{L^{\sigma}_{[0,T]}(L^{\rho})}\le
\left\|s^{\frac{1}{p}-\frac{1}{2}}  \|u\|_{L^{\A+1}}^{\A} \right\|_{L^{\sigma}_{[0,T]}}.
\end{equation*}
Applying H\"older's inequality with
\begin{equation*}
1+\frac{1}{2(\A+1)}-\frac{1}{2p}=\frac{3p-\A-1}{2p}+\frac{\A}{q}
\end{equation*}
we see that the right hand side is estimated by
\begin{equation*}
\|s ^{\frac{1}{p}-\frac{1}{2}} \|_{L^{\frac{2p}{3p-\A-1}}_{[0,T]}} \|u\|^{\A}_{L^q_{[0,T]}(L^r)}=
T^{1-\frac{\A-1}{2p}} \|u\|^{\A}_{L^q_{[0,T]}(L^r)}.
\end{equation*}
Collecting these estimates, we get
\begin{equation}
\|\mathscr{T}u\|_{X_T}\le C_1\|\phi\|_{L^p}+C_2T^{1-\frac{\A-1}{2p}} \|u\|_{X_T}^{\A} \label{mapstoitself}
\end{equation} 
as long as
\begin{equation*}
\max \left(\frac{\A-1}{2},\frac{4}{3},\frac{\A+1}{\A} \right) <p \le 2.
\end{equation*}
Similarly, the corresponding difference estimate can be obtained:
\begin{equation}
\|\mathscr{T}u_1 -\mathscr{T}u_2 \|_{X_T}
\le CT^{1-\frac{\A-1}{2p}} \|u_1 -u_2 \|_{X_T}. \label{contractionmapping}
\end{equation}
Consequently, the existence of local solution in $\mathfrak{L}^{\infty}([0,T] ; L^p)\cap L^q([0,T] ; L^r)$ follows from a
standard contraction argument: if we put
\begin{equation*}
R\triangleq 2C_1 \|\phi \|_{L^p},\quad T\triangleq C_3 \|\phi \|_{L^p}^{-\frac{2p(\A-1)  }{ 2p-\A+1 } },
\end{equation*}
then $\mathscr{T}$ maps $\mathscr{V}(T,R)$ to itself and is a contraction mapping.  Therefore, we get a local solution of (\ref{inteq})
in $X_T$.  Uniqueness and continuous dependence on data can be shown by arguing as in the proof of the difference estimate (\ref{contractionmapping}).  
Moreover, using the bilinear estimate and arguing as above, we have for $t_1,t_2 \in [0,T]$
\begin{equation*}
\|U(-t_1)u(t_1)-U(-t_2)u(t_2) \|_{L^p} \le C|t_1-t_2|^{1-\frac{\A-1}{2p}} \|u\|_{X_T}^{\A-1},
\end{equation*}
from which we see that $u$ belongs to $C_{\CS}([0,T] ; L^p(\R))$.
Finally, the Hausdorff--Young type property (\ref{lsmoothing}) follows from Lemma \ref{H--Y} and this completes the proof of Theorem \ref{LWP} for the case $1< \A\le 5$.  

\bigskip

\noindent\textbf{Case 2}. \quad We assume $3\le \A<5$.  Let $q,r$ be determined by the relation
\begin{equation}
r=2(\A-1),\quad \frac{2}{q}+\frac{1}{r}=\frac{1}{p}. \label{case1}
\end{equation}
Then it is easy to check that $(q,r)\in \mathscr{S}(p)$ for $p$ in the range given by (\ref{proofprange}) and
we get
\begin{equation*}
\|U(t)\phi \|_{L^q_{[0,T]}(L^r)}\le C\|\phi \|_{L^p}
\end{equation*}
by the generalized homogeneous Strichartz estimates.
In order to estimate the $L^{q}(L^r)$-norm of the nonlinear part, let $\gamma$ be determined by
\begin{equation*}
2+\frac{1}{p}=\frac{2}{\gamma}+\frac{\A}{2(\A-1)}.
\end{equation*}
Then the quadruple $(\beta,\kappa,\sigma, \rho)\triangleq (q,r,\gamma,\frac{2(\A-1)}{\A})$ satisfies the assumption of Proposition \ref{inhomo}.  Thus we have
\begin{eqnarray*}
\left\| \int^t_0 U(t-\tau) N(u(s))ds \right\|_{L^q_{[0,T]]}(L^r)} &\le &C \left\| |u|^{\A-1}u \right\|_{L^\gamma_{[0,T]}(L^{\frac{2(\A-1) }{\A}})}\\
&=& \|u \|^{\A}_{L^{\A\gamma}_{[0,T]}(L^{2(\A-1)})} \\
&\le & T^{1-\frac{\A-1}{2p}} \| u\|_{L^q_{[0,T]}(L^r)}^{\A}.
\end{eqnarray*}
Next we estimate $\mathfrak{L}^{\infty}_{[0,T]}(L^p)$-norm of the integral equation.  
Clearly, $\|U(t)\phi\|_{\mathfrak{L}^{\infty}_{[0,T]}(L^p)}=\|\phi \|_{L^p}$.  Now we
use our key bilinear estimate (\ref{BLWDFS}) in stead of the weighted one (\ref{WDFS}) to control $\mathfrak{L}^{\infty}(L^p)$-norm of the nonlinear term.  We apply Proposition \ref{bilinear} with
\begin{equation*}
\rho=\frac{2p}{3p-2},\quad \sigma=\frac{4}{3}.
\end{equation*}
It is easy to see that this choice of $(\rho,\sigma)$ satisfies the assumption of the proposition as long as $p>4/3$.  Hence we get
\begin{eqnarray*}
\sup_{t\in [0,T]} \left\|\int^t_0 U(-s) N(u) ds \right\|_{L^p}
&\le& C\|u^{\A-1} \|_{L^{\frac{4}{3}}_{[0,T]} (L^2)} \|u\|_{\mathfrak{L}^{\infty}_{[0,T]}(L^p)} \\
&\le & CT^{1-\frac{\A-1}{2p}} \|u\|_{L^q_{[0,T]}(L^{2(\A-1)})}^{\A-1} \|u\|_{\mathfrak{L}^{\infty}(L^p)} \\
&\le&\left\|  \int^t_0 U(t-\tau) N(u) d\tau \right\|_{\mathfrak{L}^{\infty}_{[0,T]}(L^p)}
\le CT^{1-\frac{\A-1}{2p}} \| u\|_{X_T}^{\A-1}.
\end{eqnarray*}
Collecting these estimates, we get
\begin{equation}
\|\mathscr{T}u\|_{X_T}\le C_1\|\phi\|_{L^p}+C_2T^{1-\frac{\A-1}{2p}} \|u\|_{X_T}^{\A} \label{mapstoitself}
\end{equation} 
Similarly, the corresponding difference estimate can be obtained:
\begin{equation}
\|\mathscr{T}u_1 -\mathscr{T}u_2 \|_{X_T}
\le CT^{1-\frac{\A-1}{2p}} \|u_1 -u_2 \|_{X_T}. \label{contractionmapping}
\end{equation}
Arguing as in Case 1, we also get the local well-posedness result for $3\le \A<5$.

\bigskip

\noindent\textit{Proof of Theorem \ref{SGWP}}  We assume $p=\frac{\A-1}{2}$.  We define
\begin{equation*}
X_{\infty} \triangleq \mathfrak{L}^{\infty} (\R ; L^p(\R)) \cap L^q (\R; L^r(\R) )
\end{equation*}
with the norm
\begin{equation*}
\|u\|_{X_{\infty}}\triangleq \max \left( \|u\|_{\mathfrak{L}^{\infty}(L^p)}, \|u\|_{L^q(L^r)} \right).
\end{equation*}

  Then, arguing as in the above proof one gets
\begin{equation}
\|\mathscr{T}u\|_{X_{\infty}} \le C_1 \|\phi \|_{L^p} +C_2 \|u\|_{X_{\infty}}^{\A} \label{mapstoitselfS}
\end{equation}
and
\begin{equation}
\|\mathscr{T}u_1 -\mathscr{T}u_2\|_{X_{\infty}} \le C_2 (\|u_1\|^{\A-1}_{X_{\infty}} +\|u_2|^{\A-1}_{X_{\infty}} )\|u_1 -u_2\|_{X_{\infty}}. \label{contractionmappingS}
\end{equation}
Now let $\epsilon >0$ be a positive real number and we suppose $\|\phi \|_{L^p}\le \epsilon$.  We then set
\begin{equation*}
\mathscr{V}_{\varepsilon}\triangleq \{ u\in X_{\infty}\, |\, \|u\|_{X_{\infty}} \le 2C_1\epsilon \, \}.
\end{equation*}
Then if $\epsilon$ is sufficiently small, we have by (\ref{mapstoitselfS})
\begin{equation*}
\|\mathscr{T}u \|_{X_{\infty}} \le C_1 \epsilon +C_2 (2C_1 \epsilon)^{\A} \le C1\epsilon +C_1 \epsilon =2C_1 \epsilon
\end{equation*}
for any $u \in \mathscr{V}_{\epsilon}$.  Similarly, for $u_1,u_2 \in\mathscr{V}_{\epsilon}$ we have
\begin{equation*}
\|\mathscr{T}u_1 -\mathscr{T}u_2 \|_{X_{\infty}} \le C_2 \left( (2C_1\epsilon)^{\A-1} +(2C_1 \epsilon)^{\A-1} \right) \|u_1 -u_2 \|_{X_{\infty}}
\le \frac{1}{2}\|u_1 -u_2\|_{X_{\infty}}
\end{equation*}
under some smallness assumption on $\epsilon$.
Consequently, if $\epsilon>0$ is sufficiently small we get a global solution such that
\begin{equation*}
\max \left( \|U(-t)u(t) \|_{L^{\infty} (L^p)}  , \|u\|_{L^q(L^r)} \right) \le 2C_1 \epsilon.
\end{equation*}
Moreover, by Proposition \ref{bilinear}, we have
\begin{equation*}
\sup_{t\in \R} \left\| \int^t_0 U(-s)N(u(s))ds \right\|_{L^p} \le C \|u\|_{L^q(L^r)}^{\A-1} \|U^{-1}u\|_{L^{\infty}(L^p)}<\infty,
\end{equation*}
from which it follows that
\begin{equation*}
M\triangleq \sup_{t\in \R} \|U(-t)u(t)\|_{L^p}<\infty
\end{equation*}
and the decay estimate
\begin{equation*}
\|u(t) \|_{L^{p'}} \le (4\pi |t|)^{-(\frac{1}{p}-\frac{1}{2})} \times M,\quad t\neq 0.
\end{equation*}
The other assertions in Theorem \ref{SGWP} can also be obtained by arguing as in the proof of Theorem \ref{LWP}.

\begin{rem}
The pair $(q,r)$ in the above proof belongs to $\mathscr{S}(p).$  However,
note also that $(q,r)$ may not be in $\widehat{\mathscr{S}}(p)$.  For example take $\A \in (3, 4)$ and $p$ sufficiently close to $4/3$.
\end{rem}

\begin{rem}

We emphasize that we cannot use (\ref{WDFS}) to obtain the existence results in Theorem \ref{LWP} and \ref{SGWP} for all $\A$'s in the range $(1,5)$.  In particular, it is easy to see that the existence of the weight
$s^{1/p-1/2}$ in (\ref{WDFS}) would be an obstacle to the construction of small global solutions in $L^q(\R ;L^r)$ spaces unless $(\A,p)=(5,2)$.
\end{rem}

\section{Proof of Theorem \ref{LGWP}} \label{secGWP}
In this section we prove the global well-posedness result for (\ref{NLS}) with $N(u)=|u|^{\A-1}u$ for large data $\phi \in L^p$.  Our strategy is inspired by the works\cite{Carles}, \cite{107H2}, where the authors prove global well-posedness result in $L^2\cap X$ for some non-$L^2$-based space $X$ but here we do not assume
square integrability of data $\phi$.  Hereafter, we only consider the case $t>0$.  The case of negative time can be treated in the similar manner.  Roughly speaking, the global well-posedness is achieved in two steps.  The first step is to construct a global solution in some suitable space-time Lebesgue space $L^q_{loc}(L^r)$.  
The existence of global solutions of this kind is known (see \cite{107T2},\cite{107}).  This is proved by means of a data-decomposition argument developed by Vargas and Vega \cite{VV} and Bourgain \cite{Bourgain}. Here our analysis is based on a global existence result in \cite{107}.  For convenience, we give 
a sketch of its proof  in appendix.  In the second step, we prove that the global solution $u$ given by the previous step satisfies $U(-t)u(t)|_{[-T,T]}\in C([-T, T] ; L^p)$ for any large $T>0$.  This is
a consequence of the local result (Theoerm \ref{LWP}) and a key proposition which asserts that $U(-t)u(t)|_{[0,\varepsilon)} \in C([0,\varepsilon) ; L^p)$ for at least sufficiently small $\varepsilon>0$ implies
$U(-t)u(t)|_{[0,\infty)} \in C([0,\infty) ; L^p)$.  This key proposition is presented and proved in subsection \ref{Sapriori}.

\bigskip
\noindent\textit{List of notation for Section \ref{secGWP}}.
\begin{itemize}
\item
For a given $r \ge 2$ and $p\in [1,2]$, $Q_p(r)$ is determined by the scaling relation
\begin{equation*}
\frac{2}{Q_p(r)}+\frac{1}{r}=\frac{1}{p}.
\end{equation*}
\item
Let $p\in [1,2]$ and let $u :\R\times \R \to \C$ be a global solution to (\ref{NLS}) such that
\begin{equation*}
u\in L^{Q_{p}(2\A-2)}_{loc}(\R ; L^{2\A-2}(\R)).
\end{equation*}
We define
\begin{equation}
T_{\max}(u)\triangleq \sup \{ T>0,\,\, |\,\, u|_{[0,T]\times \R} \in C_{\CS}([0,T] ;L^p(\R)) \,  \,\}.
\end{equation}

\end{itemize}

\subsection{Construction of a global solution via splitting argument}\label{Ssplitting}

In this subsection, we briefly review the global existence result in \cite{107} without property (\ref{ipersistency}) for data
$\phi \in L^p$.  

\begin{defn}
Let $p\in (1,2)$, $1<\A<5$, and $\gamma>0$.  Let $\mathcal{A}(p,\A+1,\gamma)$ be the set of functions in $\mathcal{S}'(\R)$ such that: there exist $(\varphi_N)_{N>1} \subset L^2(\R)$ and $(\psi_N)_{N>1}\subset \mathcal{S}'(\R)$ such that
\begin{equation}
\phi=\varphi_N+\psi_N,\quad \forall N>1, \label{decomp1}
\end{equation}
\begin{equation}
C_0^{-1} N^{\gamma} \le \|\varphi_N \|_{L^2} \le C_0 N^{\gamma},\quad \forall N>1,  \label{decomp2}
\end{equation}
and
\begin{equation}
\|U(t)\psi_N\|_{L_{\R}^{Q_p(\A+1)}  (L^{\A+1}) }\le C_0 N^{-1}, \label{decomp3}
\end{equation}
where $C_0>0$ is a constant independent of $N$.

\end{defn}

By means of a splitting argument, we get a local existence result for data in $\mathcal{A}(p,\A+1,\gamma)$ which 
leads directly to a global one.
\begin{prop}\label{gexisprop1}
Assume that
\begin{equation}
2>p >\max \left(\frac{2\A}{5},\frac{\A+1}{\A}\right). \label{pranggexis}
\end{equation}
Then, there is $N_0>1$ such that for each $N>N_0$ and $\phi \in \mathcal{A}(p,\A+1,\gamma)$ there exists a local solution $u:[0,T_N] \times \R\to \C$ of (\ref{NLS}) such that
\begin{equation*}
u \in C ([0,T_N] ; L^2(\R))\cap L^{Q_2(\A+1)}([0,T_N] ; L^{\A+1}) 
+L^{Q_p(\A+1)} ([0,T_N] ; L^{\A+1} )
\end{equation*}
where $T_N>0$ is of the form
\begin{equation*}
T_N \sim N^{1+\frac{2}{5-\A} (4-2\A-\frac{\A-1}{p})\gamma}.
\end{equation*}
Moreover, the following uniqueness assertion holds: let $u^{(N)}$ be the above solution $u$ given for each fixed $N>N_0$.  Then 
\begin{equation*}
u^{(N_1)} (t,x) =u^{(N_2)}(t,x),\quad a.a. (t,x) \in [0, \min (T_{N_1}, T_{N_2}) ] \times \R
\end{equation*}
for any $N_1,N_2 >N_0$.
\end{prop}

It is easy to see that $\lim_{N\to \infty} T_N=\infty$ under some additional assumptions on $\A,\gamma,p$ which means the solution can reach arbitrarily large time.  To be more precise, this existence and uniqueness assertion yields the following global existence result
for data $\phi \notin L^2$:
\begin{cor} \label{gexiscor1}
In addition to {\rm (\ref{pranggexis})}, assume moreover that
\begin{equation}
\frac{2}{5-\A} (4-2\A-\frac{\A-1}{p})\gamma >-1.
\end{equation}
Then for any $\phi \in \mathcal{A}(p,\A+1,\gamma)$, there exists a unique global solution $u$ of the form
\begin{equation*}
u \in C(\R ; L^2(\R)) \cap L_{loc}^{Q_2(\A+1)}(\R; L^{\A+1}) 
+L_{loc}^{Q_p(\A+1)} (\R; L^{\A+1} ) \subset L_{loc}^{Q_p(\A+1)} (\R; L^{\A+1} ).
\end{equation*}

\end{cor}

This global existence theorem can apply to the case of $L^p$ Cauchy data.  This is due to the following lemma:
\begin{lem}\label{pdecomp}{\rm (\cite{107})}
Assume that
\begin{equation*}
2>p >\max \left(\frac{2\A}{5}, \frac{\A+1}{\A} \right).
\end{equation*}
Let $p_0$ be such that
\begin{equation*}
p>p_0 >\max \left(\frac{2\A}{5}, \frac{\A+1}{\A} \right).
\end{equation*}
Then
\begin{equation*}
L^p\setminus L^2 \subset \mathcal{A}(p_0 ,\A+1 ,\gamma)
\end{equation*}
with 
\begin{equation*}
\gamma=\frac{\frac{1}{p}-\frac{1}{2}}{\frac{1}{p_0}-\frac{1}{p}}.
\end{equation*}
\end{lem}

Proposition \ref{gexisprop1} and Corollary \ref{gexiscor1} yields the
global existence theorem for large data $\phi \in L^p$ which is our first key proposition.
\begin{prop} \label{gexiscor2}
Assume that
\begin{equation*}
\max \left( \frac{(\A-1)(\A+3)}{2\A^2+2\A-4}, \frac{(\A-1)(3\A+5)}{2(\A^2-2\A+5)} \right)<p\le 2.
\end{equation*}
Then for any $\phi \in L^p$ there is a unique global solution $u$ of the form
\begin{equation*}
u \in L_{loc}^{Q_{p_0}(\A+1)}(L^{\A+1})
\end{equation*}
for some $p_0$ such that
\begin{equation*}
p>p_0 >\max \left(\frac{2\A}{5}, \frac{\A+1}{\A} \right).
\end{equation*}

\end{prop}

\subsection{Key Lemma}\label{Sapriori}
Now we state the second key proposition to Theorem \ref{LGWP}.

\begin{prop}\label{GWPKEY}
Let
\begin{equation*}
\A\ge 3,\quad \max \left( \frac{\A-1}{2}, \frac{4}{3}\right) <p \le 2,
\end{equation*}
and let $u$ be a global solution to (\ref{NLS}) such that
\begin{equation}
u\in L^{Q_{p_0}(2\A-2)}_{loc}(\R ; L^{2\A-2} (\R) ) \label{llebesgue}
\end{equation}
for some $p_0 \in (\frac{\A-1}{2}, 2]$.  Suppose that
\begin{equation}
U(-t)u(t)|_{[0,\varepsilon) \times \R} \in C([0,\varepsilon) \, ; L^p (\R) ). \label{persistencyhyp}
\end{equation}
for some $\varepsilon >0$.
Then
\begin{equation*}
U(-t)u(t)|_{[0,\infty) \times \R} \in C([0,\infty) ; L^p (\R)).
\end{equation*}

\end{prop}

\medskip

To prove the proposition, we need a few lemmata.  We first need a slight variant of the key bilinear estimate in Section \ref{SectionKey}.
\begin{lem}\label{apriorilem}
Let $\A,p,p_0$ be such that
\begin{equation*}
3 \le \A,\quad \max \left(\frac{\A-1}{2},\frac{4}{3} \right)<p_0<p\le 2.
\end{equation*}

Then there exists $\gamma \in [\frac{4}{5-\A},\infty)$ depending only on $p$ such that the following estimate holds for any $t>0$:
\begin{equation*}
\left\| \int^t_0 U(-s)N(u(s)) ds \right\|_{L^p} \le C \|u \|_{L_{[0,t]}^{Q_{p_0}(2\A-2)} (L^{2\A-2})} \|U^{-1} u \|_{L^{\gamma}_{[0,t]}(L^p)}.
\end{equation*}

\end{lem}

\begin{proof}
We put
\begin{equation*}
\rho=\frac{2p}{3p-2},\quad \sigma=\frac{4}{3}
\end{equation*}
in (\ref{WDFS}) to obtain
\begin{eqnarray*}
\left\| \int^t_0 U(-s)N(u(s)) ds \right\|_{L^p} &\le &C\| s^{\frac{1}{2}-\frac{1}{p}} N(u(s)) \|_{L^{\sigma}_{[0,t]} (L^{\rho})} \\
&\le & \left\| \| |u(s)|^{\A-1} \|_{L^2} \|U(-s)u(s) \|_{L^p} \right\|_{L^{\frac{4}{3}}_{[0,t]}} \\
&\le & C \|u \|_{L_{[0,t]}^{\frac{4\gamma (\A-1)}{3\gamma -4}} (L^{2(\A-1)})} \|U^{-1} u \|_{L^{\gamma}_{[0,t]}(L^p)}.
\end{eqnarray*}

We get the assertion since
\begin{equation*}
\lim_{\gamma \searrow \frac{4}{5-\A}} \frac{4\gamma (\A-1)}{3\gamma -4}=Q_2(2\A-2),\qquad \lim_{\gamma \to \infty} \frac{4\gamma (\A-1)}{3\gamma -4}=Q_{\frac{\A-1}{2}}(2\A-2).
\end{equation*}
\end{proof}

  The following
lemma is obtained arguing as in the proof of the difference estimate in $L^q(L^r)$-spaces in Section \ref{SLWP}.
\begin{lem}\label{GWPunique}
\item
Assume that
\begin{equation*}
p_0 >\max \left(\frac{\A-1}{2}, \frac{2(\A-1)}{\A} \right).
\end{equation*}
Let $T\ge 0$ and $\varepsilon_j>0, \,j=1,2>0$ and let $u_j :[0,T+\varepsilon_j]\times \R \to\C,\,\,j=1,2$ be two solutions to (\ref{NLS}) such that
\begin{equation*}
u_1(T)=u_2(T)
\end{equation*}
and
\begin{equation*}
u_j \in L^{Q_{p_0}(\A-2)}([T,T+\varepsilon_j] ; L^{2\A-2}(\R)).
\end{equation*}
Then
\begin{equation*}
u_1(t,x)=u_2(t,x),\quad {\rm a.a.}\,(t,x) \in [T,T+\min(\varepsilon_1,\varepsilon_2)] \times \R.
\end{equation*}
\end{lem}

\bigskip

\begin{lem}\label{blowupalt}
Assume that $\A\ge 3$ and
\begin{equation}
\max \left(\frac{\A-1}{2},\frac{2(\A-1)}{\A}, \frac{\A+1}{\A} \right) <p \le 2. \label{blowupcondp}
\end{equation}

Let $u$ be a global solution to (\ref{NLS}) such that
\begin{equation*}
u \in L_{loc}^{Q_p(2\A-2)} (\R ; L^{2\A-2}(\R)).
\end{equation*}
Suppose that $0<T_{\max}(u)<\infty$.  Then
\begin{equation}
\lim_{t\nearrow T_{\max}(u)} \|U(-t)u(t) \|_{L^p} =\infty. \label{blowup}
\end{equation}

\end{lem}

\begin{proof}
Following \cite{Zhou} we introduce the linear transformation $v(t)\triangleq U(-t)u(t)$.  Then $v$ solves the integral equation
\begin{equation}
v(t)=\phi +i\int^t_0 U(-s) N(U(s)v(s)) ds \label{vint}
\end{equation}
and Theorem \ref{LWP} can be rephrased as follows:\\
\noindent\textbf{Claim.}
\textit{Assume {\rm(\ref{blowupcondp})} and let $M>0$.  Then for any $\phi \in L^p$ with $\|\phi\|_{L^p}\le M$ there are $T\triangleq T(M)>0$ and a unique local solution to {\rm(\ref{vint}) }such that
\begin{equation*}
v\in C([0,T] ;L^p (\R)) \cap \{ v:[0,T] \times \R \to \C \,|\, U(t)v(t) \in L^{Q_p(2\A-2)}([0,T] ; L^{2\A-2}(\R)) \}.
\end{equation*}
}

\medskip
We continue with the proof of Lemma \ref{blowupalt}.  Suppose that (\ref{blowup}) does not hold.  Then there are $M_0>0$ and $(T_k)_{k\in \N} \subset (0,T_{\max}(u))$ such that
\begin{equation}
t_k \nearrow T_{\max}(u),\,\,(\text{as}\,\,k \to \infty),\qquad \text{and}\qquad \|v(t_k) \|_{L^p} \le M_0,\quad \forall k\in\N.
\end{equation}
Therefore, applying the above claim, the solution $v$ to (\ref{vint}) can be extended to the time $t=t_k+T(M_0)$ for each fixed $k \in \N$.  
Since $T(M_0)$ is independent of $k$, we have $T_{\max}(u)<t_k +T(M_0)$ for sufficiently large $k$.  Now let $\tilde{v_k}:[0,t_k+T(M_0)] \to \C$ 
be the extended solution given for $k\in \N$.  Then by Lemma \ref{GWPunique}, $u_k(t)\triangleq U(t)v_k(t)$ coincides with $u$ on $[t_k, t_k +T(M_0)]$.  This implies that 
$U(-t)u(t) \in C([0, T_{\max}(u)+\varepsilon] ; L^p)$ for some $\varepsilon>0$.  Absurdity.

\end{proof}

\medskip
\noindent Now we prove our second key proposition.\\
\noindent\textit{Proof of Proposition \ref{GWPKEY}}.\, By the assumption (\ref{persistencyhyp}), $T_{\max}(u)>0$.  In addition, we suppose that $T_{\max}(u)<\infty$.  We start from the corresponding integral equation (\ref{inteq}).  Take
$T\in (0, T_{\max}(u))$ and fix it.  Taking the norm $\|U^{-1}\cdot \|_{L^p}$ of both sides and then
using Lemma \ref{apriorilem}, we get for any $t \in [0, T]$
\begin{eqnarray*}
\|U(-t)u(t) \|_{L^p} &\le & \|\phi \|_{L^p} +\left\|\int^t_0 U(-s) N(u(s))ds \right\|_{L^p}\\
&\le & \|\phi\|_{L^p} +C\|u\|_{L^{Q_{p_0}(2\A-2)}_{[0,t]}(L^{2\A-2})} \|U^{-1}u \|_{L^{\gamma}_{[0,t]}(L^p)}
\end{eqnarray*}
for some $\gamma>0$.  By (\ref{llebesgue})
\begin{equation*}
 \|u \|_{L_{[0,t]}^{Q_p(2\A-2)}}\le  \|u \|_{L_{[0,T_{\max}(u)]}^{Q_p(2\A-2)}}=C<\infty
\end{equation*}
for some $p_0 \in (\frac{\A-1}{2},2]$.
We set
\begin{equation*}
\omega(t) \triangleq \|U(-t)u(t)\|_{L^p}^{\gamma}.
\end{equation*}

Then by the above estimate, we have
\begin{equation*}
\omega (t) \le C_1+C_2 \int^t_0 \omega(s) ds
\end{equation*}
for any $t\in [0,T]$.
By Grownwall's lemma, 
\begin{equation*}
\omega (t) \le C\exp (C_2 t),\quad \forall t\in [0,T]
\end{equation*}
from which we easily deduce that
\begin{equation*}
\sup_{t \in [0,T_{\max}(u))} \|U(-t)u(t)\|_{L^p} <\infty
\end{equation*}
since $T\in (0,T_{\max}(u))$ is arbitrary.
This contradicts Lemma \ref{blowupalt}.  Hence $T_{\max}(u)=\infty$. \qed

\subsection{Proof of Theorem \ref{LGWP}}

\begin{lem}\label{GWPunique2}
Let $\A>2$ and let $p,p_0$ be such that
\begin{equation*}
\max\left( \frac{\A-1}{2} , \frac{2\A-2}{2\A-3} \right) <p\le 2
\end{equation*}
and
\begin{equation*}
\max\left( \frac{\A-1}{2} , \frac{2\A-2}{2\A-3} \right) <p_0< p.
\end{equation*}

Let $u$ be a global solution to (\ref{NLS}) such that
\begin{equation}
u(0)=\phi \in L^p (\R)
\end{equation}
and
\begin{equation}
u \in L_{loc}^{Q_{p_0}(\A+1)}(\R ; L^{\A+1}(\R) ).
\end{equation}
Then
\begin{equation}
u \in L_{loc}^{Q_{p_0}(2\A-2)}(\R ; L^{2\A-2}(\R) ).
\end{equation}

\end{lem}

\begin{proof}
Let $T>0$.  
We estimate the corresponding integral equation (\ref{inteq}) by means of the generalized Strichartz estimates.  Observe that $(Q_p(2\A-2), 2\A-2) \in \mathscr{S}(p)$ if
\begin{equation}
\A>2,\quad p>\frac{2\A-2}{2\A-3}. \label{luniquecond}
\end{equation}
Let $\gamma$ be determined by the relation
\begin{equation*}
2+\frac{1}{p_0}=\frac{2}{\gamma}+\frac{\A}{\A+1}.
\end{equation*}
Then the quadruple $(\beta,\kappa,\sigma,\rho)=(Q_p(2\A-2),2\A-2,\gamma, (\A+1)/\A)$ fulfills the condition of Proposition \ref{inhomo} if $p>(\A+1)/\A$ in addition to (\ref{luniquecond}).  Therefore,
\begin{eqnarray*}
\| u\|_{L_{[0,T]}^{Q_{p_0}(2\A-2)}(L^{2\A-2})} &\le & \| U(t)\phi \|_{L_{[0,T]}^{Q_{p_0}(2\A-2)}(L^{2\A-2})} +\left\| \int^t_0 U(t-s)N(u(s))ds \right\|_{ L_{[0,T]}^{Q_{p_0}(2\A-2)}(L^{2\A-2})  }\\
 &\le & T^{\frac{1}{Q_{p_0}(2\A-2)}-\frac{1}{Q_p(2\A-2)}}\| U(t)\phi \|_{L_{[0,T]}^{Q_{p}(2\A-2)}(L^{2\A-2})} \\
&& \qquad +\left\| \int^t_0 U(t-s)N(u(s))ds \right\|_{ L_{[0,T]}^{Q_{p_0}(2\A-2)}(L^{2\A-2})  }\\
&\le & CT^{\frac{1}{Q_{p_0}(2\A-2)}-\frac{1}{Q_p(2\A-2)}}\|\phi \|_{L^p} +C\|N(u)\|_{L_{[0,T]}^{\gamma}(L^{\frac{\A+1}{\A}}) } \\
&\le & CT^{\frac{1}{2}\left(\frac{1}{p_0}-\frac{1}{p}\right)} \|\phi \|_{L^p} +CT^{1-\frac{\A-1}{2p_0}} \|u\|_{L^{Q_{p_0}(\A+1)} (L^{\A+1})}^{\A},
\end{eqnarray*}
where $\gamma$ is such that
\begin{equation*}
2+\frac{1}{p_0}=\frac{2}{\gamma}+\frac{\A}{\A+1}.
\end{equation*}
\end{proof}

As consequence of Theorem \ref{LWP}, Lemma \ref{GWPunique}, Lemma \ref{GWPunique2}, and Proposition \ref{gexiscor2}, we get the following
global existence theorem with (\ref{ipersistency}) at least for a finite time.
\begin{cor}\label{GECOR}
Let $\A\ge 3$ and
\begin{equation*}
\max \left( \frac{(\A-1)(\A+3)}{2\A^2+2\A-4}, \frac{(\A-1)(3\A+5)}{2(\A^2-2\A+5)} \right)<p \le 2
\end{equation*}
Then for any $\phi \in L^p$ there exists a unique global solution $u$ to (\ref{NLS}) such that
\begin{equation*}
u \in L^{Q_{p_0}(2\A-2)}_{loc}(\R ; L^{2\A-2}(\R))
\end{equation*}
for some $p_0$ satisfying
\begin{equation*}
\max \left( \frac{2\A}{5}, \frac{\A+1}{\A} \right) <p_0 <p
\end{equation*}
and such that
\begin{equation*}
u|_{[0,T] \times \R} \in C_{\CS} ([-T,T] ; L^p (\R) )
\end{equation*}
for some $T\sim \|\phi \|_{L^p}^{-\frac{2p(\A-1)}{2p-\A+1}}$.

\end{cor}

\medskip
\noindent\textit{Proof of Theorem \ref{LGWP}}.\quad We get the assertion of
our main result on global well-posedness for large data by combining Corollary \ref{GECOR} with Proposition \ref{GWPKEY}.

 \section{Proof of other results  }

 \subsection{Proof of Theorem \ref{hatglobal}}
 
 We begin by recalling the local well-posedness and global existence result given by \cite{107T}:
 \begin{prop}\label{Hatexistence}
 Assume that $1<\A<5$ and
 \begin{equation*}
 2\le p<   \min \left( \frac{2(\A+1)}{\A-1}, \A+1 \right).
 \end{equation*}
Let $M>0$. Then for any $\phi \in \Lp$ with $\|\phi\|_{\Lp}\le M$ there are 
$T=T(M)>0$ and a unique local solution $u$ to (\ref{NLS}) such that
\begin{equation*}
u \in C([-T,T] ; \Lp (\R)) \cap L^q([-T,T] ; L^{\A+1}(\R) )
\end{equation*}
 where $q$ is determined by
 \begin{equation*}
 \frac{2}{q}+\frac{1}{\A+1}=\frac{1}{p}.
 \end{equation*}
 Moreover, if the nonlinearity is give by $N(u)=|u|^{\A-1}u$ and
 \begin{equation*}
 2\le p <\min \left( \A+1, \frac{3\A+5}{2\A} \right),
 \end{equation*}
 the local solution extends to a global one such that
 \begin{equation*}
 u \in C(\R; L^2(\R))\cap L^{\frac{4(\A+1)  }{ \A-1 }}_{loc} (L^{\A+1}(\R))
 +L^Q_{loc}([0,\infty) ; L^{\A+1} (\R))
 \end{equation*}
 where $Q>q$ is a sufficiently large positive number determined depending only on $p,\A$.
 Furthermore, uniqueness holds in the space $L^{\frac{4(\A+1)  }{ \A-1 }}_{loc} (\R;L^{\A+1} (\R))$.
 \end{prop}

 This result assures the global existence for $\phi \in \Lp$.  It remains to
 show the global-in-time persistence property
 \begin{equation}
 u\in C([0,\infty) ; \Lp) \label{Hatpersistencyg}
 \end{equation}
 for the solution (as earlier, it is sufficient to consider only positive time).  In view of Proposition \ref{Hatexistence}, the
 global-in-time persistence property (\ref{Hatpersistencyg}) is a direct consequence of
 the following proposition:

 \begin{prop}\label{hatapriori}
 Let $2<\A <5$ and $2\le p <4$.  Let $u:[0,\infty) \times \R \to \C$ be a global solution such that
 \begin{equation*}
u\in L_{loc}^{\frac{4(\A+1)}{\A-1}} ([0,\infty) ; L^{\A+1}(\R) )
 \end{equation*}
 and
 \begin{equation*}
 u|_{[0,\varepsilon)\times \R} \in C([0,\varepsilon) ; \Lp(\R))
 \end{equation*}
 for some $\varepsilon>0$.
 Then
 \begin{equation*}
u \in C([0,\infty) ;\Lp(\R)).
 \end{equation*}
 
 \end{prop}

\noindent\textit{Proof}.
We proceed in two steps:\\
\textbf{Step 1}.\quad Assume moreover that 
\begin{equation*}
u|_{[0,T_0) \times \R} \in C([0,T_0) ; \Lp (\R))
\end{equation*}
for some $T_0>0$.
Then 
\begin{equation}
\sup_{t\in [0,T_0)} \|u(t) \|_{\Lp}<\infty. \label{hatstep1}
\end{equation}

\textit{Proof of Step 1}.\quad Let $t\in [0,T_0)$.  We apply Corollary \ref{DFS}
to obtain
\begin{equation*}
\left\| \int^t_0 U(t-\tau) (|u|^{\A-1} u) d\tau \right\|_{\Lp}
\le \| |u|^{\A-1}u \|_{L_{[0,t]}^{\sigma}(L^{\rho})},
\end{equation*}
 where
\begin{equation*}
\sigma=\frac{4}{3},\quad \rho=\frac{2p}{p+2}.
\end{equation*}
For any $s\in [0,t]$, we have
\begin{equation*}
\left\||u(s)|^{\A-1} u(s) \right\|_{L^{\frac{2p}{p+2}}} \le \left\| |u(s)|^{\A-1}\right\|_{L^2} \|u\|_{L^p}
\le \|u(s)\|_{L^{2(\A-1)}}^{\A-1} \|u(s)\|_{\Lp}.
\end{equation*}
Thus
\begin{equation}
\left\| \int^t_0 U(t-\tau) (|u|^{\A-1} u) d\tau \right\|_{\Lp}
\le C\|u\|_{L^{\frac{4(\A-1)}{\A-2}  }_{[0,t]}(L^{2(\A-1)})} \|u \|_{L_{[0,t]}^{\frac{4 }{5-\A }} (\Lp)}
\label{hatapriori2}
\end{equation}
and
note that $(\frac{4(\A-1)}{\A-2}, 2(\A-1))$ is admissible if $\A\ge 2$.  To complete the proof we need
two elementary lemmata:
\begin{lem} \label{FSPA}
For any $2\le q,r,p \le \infty$ such that
\begin{equation*}
\frac{2}{q}+\frac{1}{r}=\frac{1}{p}
\end{equation*}
the estimate
\begin{equation*}
\|U(t)\phi \|_{L^q(L^r)} \le C\|\phi \|_{\Lp}
\end{equation*}
holds true.

\end{lem}
\begin{proof}
The estimate follows from interpolation from the standard Strichartz estimates and the trivial 
case $p=q=r=\infty$.  See \cite{Kato} for details.

\end{proof}

\begin{lem} 
Let $\A>2$ and let $u$ be a global solution to (\ref{NLS}) such that
\begin{equation*}
u \in L_{loc}^{ \frac{4(\A+1)}{(\A-1)}          }([0, \infty) ; L^{ \A+1} ).
\end{equation*}
and $u(0) =\phi \in \Lp$.
Then
\begin{equation*}
u \in L_{loc}^{\frac{4(\A-1)}{(\A-2) }  }([0, \infty) ; L^{ 2(\A-1)} ).
\end{equation*}
\end{lem}

\begin{proof}
Fix $T>0$.  By Lemma \ref{FSPA} and the admissible version of Strichartz esitmates for the inhomogeneous equation, we have
\begin{eqnarray*}
\|u\|_{L^{\frac{ 4(\A-1) }{\A-2  }  }_{[0,T]  }  (L^{2(\A-1)}   ) }
&\le & \| U(t) u(0)\|_{L^{\frac{ 4(\A-1) }{\A-2  }  }_{[0,T]  }  (L^{2(\A-1)}   ) }+ \left\|        \int^t_0 U(t-\tau) N(u(\tau))
\right\|_{L^{\frac{ 4(\A-1) }{\A-2  }  }_{[0,T]  }  (L^{2(\A-1)}   ) } \\
&\le & T^{\frac{1}{2}-\frac{1}{p}} \|U(t) u(0) \|_{ L^{ \frac{2p(\A-1)}{2(\A-1)-p} }_{[0,T]}  
(L^{2(\A-1)})  } + C\left\| |u|^{\A-1}u      \right\|_{L_{[0,T]}^{ \frac{4(\A+1)}{3\A+5}  }  (L^{\frac{\A+1}{\A}})      } \\
&\le & CT^{\frac{1}{2}-\frac{1}{p}} \|u(0)\|_{\Lp} + C\|u\|_{L_{[0,T]}^{ \frac{4\A(\A+1)}{3\A+5} }  (L^{\A+1})      }^{\A} \\
&\le & CT^{\frac{1}{2}-\frac{1}{p}} \|u(0)\|_{\Lp} + CT^{\frac{5-\A}{4}}\|u\|^{\A}_{L_{[0,T]}^{\frac{4(\A+1)}{\A-1}  } (L^{\A+1})          }.
\end{eqnarray*}

\end{proof}

\bigskip
Now we go back to (\ref{hatapriori2}) and continue with the proof of Proposition \ref{hatapriori}.  The above lemma tells us that
$u \in L^{\frac{4(\A-1)}{\A-2}}_{loc}(L^{2(\A-1)})$.  Therefore for $u\in [0,T_0)$, we have
\begin{equation*}
\|u(t)\|_{\Lp} \le \|\phi \|_{\Lp} +C\|u\|_{L^{\frac{4(\A-1)}{\A-2}  }_{[0,T_0]}(L^{2(\A-1)})} \|u \|_{L_{[0,t]}^{\frac{4 }{5-\A }} (\Lp)}
\end{equation*}
Now we may write
\begin{equation*}
\|u(t)\|^{\frac{4}{5-\A}}_{\Lp} \le C +C_{T_0} \int^t_0 \|u(s)\|_{\Lp}^{\frac{4}{5-\A}} ds.
\end{equation*}
Then by Grownwall's lemma, we have
\begin{equation*}
\|u(t)\|_{\Lp} \le C\exp(CT_0).
\end{equation*}
Consequently, we have
\begin{equation*}
\sup_{t\in [0,T_0)} \|u(t)\|_{\Lp}<\infty.
\end{equation*}

\noindent\textbf{Step 2}.  Set
\begin{equation*}
\widehat{T}_{\max}(u) \triangleq \sup\{ T>0\,| \, u|_{[0,T)\times \R} \in C([0,T) ;\Lp) \}.
\end{equation*}
Then $\widehat{T}_{\max}(u)=\infty$.

\noindent\textit{Proof of Step 2}\quad The estimate (\ref{hatstep1}) has been proved and the assertion follows from a standard argument of blow-up alternative as in
the proof of Theorem \ref{LGWP}.
  \begin{lem}
 Let $2\le p<4$ and let $(\rho,\sigma)$ be such that $(\rho',\sigma')\in \widehat{
 \mathscr{S}}(p')$.
 
 \end{lem}

 We put
 \begin{equation*}
 \rho=\frac{2p}{p+2}+\varepsilon_1
 \end{equation*}
 where $\varepsilon_j>0$ is taken sufficiently small.

\subsection{Proof of Corollary \ref{scatter}}
The corollary is proved in a very standard manner.  Arguing as in the proof of Theorem \ref{SGWP}
\begin{equation}
\|U(-t_1)u(t_1) -U(-t_2)u(t_2) \|_{L^p}\le C\|u\|_{L^{q}_{[t_1,t_2]}(L^{2(\A-1)})}^{\A-1} \|U^{-1} u\|_{L^{\infty}(L^p)}
\end{equation}
Now since $u \in L^q( \R ;L^{2(\A-1)})$ and $M\triangleq\sup_{t\in \R}\|U(-t)u(t)\|_{L^p}<\infty$, we have
\begin{equation*}
\lim_{t_1,t_2 \to \infty} \|u\|_{L^{q}_{[t_1,t_2]}(L^{2(\A-1)})}^{\A-1} \|U^{-1} u\|_{L^{\infty}(L^p)} \le M\lim_{t_1,t_2 \to \infty} \|u\|_{L^{q}_{[t_1,t_2]}(L^{2(\A-1)})}^{\A-1}=0,
\end{equation*}
which implies $U(-t)u(t)$ is a Cauchy sequence in $L^p$ concerning the limit $t\to \infty$.  Hence the existence of the scattering state $\phi_+\in L^p$.

\appendix
\section{Sketch of Proof of Proposition \ref{gexisprop1}}

\renewcommand{\theequation}{\Alph{section}.\arabic{equation}}

\quad Before proceeding to the proof we prepare preliminary estimates.  We set
\begin{equation*}
G(v,w_1,w_2) \triangleq |v+w_1|^{\A-1}(v+w_1) -|v+w_2|^{\A-1}(v+w_2).
\end{equation*}
Note that $G(v,w,0)=|v+w|^{\A-1}(v+w)-|v|^{\A-1}v$.
Hereafter, we write $r=\A+1$ for readability.  Since $(Q_2(r), r)$ is admissible, we have
\begin{equation}
\|U(t)\varphi \|_{L^{Q_2(r)}(\R;L^{r})} \le C\|\varphi \|_{L^2}.
\label{L2StrS3}
\end{equation}
Moreover, if $p$ is in the range given by (\ref{pranggexis}), the estimates
\begin{eqnarray}
&&\left\| \int^t_0 U(t-s) G(v,w_1,w_2) ds \right\|_{L^{Q_p(r)}_{[0,T]}(L^{r})}  \label{NE1} \\
&\le & C\biggl[ T^{1-\frac{\A-1}{4}} \|v\|_{L_{[0,T]}^{Q_2(r)}(L^r)}^{\A-1} 
+T^{1-\frac{\A-1}{2p}} \sum_{j=1}^2\|w_j \|_{L^{Q_p(r)}_{[0,T]}(L^r)}^{\A-1} \biggr] 
\times \|w_1-w_2 \|_{L^{Q_p(r)}_{[0,T]}(L^r)}  \notag 
\end{eqnarray}
and
\begin{eqnarray}
&&\left\| \int^t_0 U(t-s) G(v,w,0) ds \right\|_{L^{\infty}_{[0,T]}(L^2)}  \label{NE2} \\
&\le & C\biggl[ T^{1-\frac{\A-1}{4}  -\frac{1}{2}\left(\frac{1}{p}-\frac{1}{2}\right)} \|v\|_{L_{[0,T]}^{Q_2(r)}(L^r)}^{\A-1} \|w\|_{L^{Q_p(r)}_{[0,T]}(L^r)}
+T^{1-\frac{\A-1}{2p}-\frac{1}{2}\left(\frac{1}{p}-\frac{1}{2}\right)} \|w\|_{L^{Q_p(r)}_{[0,T]}(L^r)}^{\A} \biggr] \notag
\end{eqnarray}
hold true for any $T>0$.

Now we establish a solution to (\ref{NLS}) for $\phi \in \mathcal{A}(p,r,\gamma)$. We fix $N$ and we first consider the Cauchy problem
\begin{equation}
iv_t +v_{xx}+|v|^{\A-1}v=0,\quad v|_{t=0}=\varphi_N. \label{CP2}
\end{equation}
Since $\varphi_N \in L^2$, we already know that the unique global solution $v\in C([0,T] ; L^2)\cap L^{Q_2(r)}([0,T] ;L^r)$
exists and has the $L^2$-conservation law
\begin{equation}
\|v(t,\cdot) \|_{L^2} =\|\varphi_N\|_{L^2},\quad \forall t\in \R. \label{conservation}
\end{equation}
Next, we define
\begin{equation}
\delta_N=(M\cdot N)^{-\frac{4(\A-1)}{5-\A}\gamma } \label{deltadef}
\end{equation}
for an absolute constant $M>0$ which will be determined sufficiently large.  It is not difficult to show that the estimate
\begin{equation}
\|v \|_{L^{Q_2(r)}_{[0,\delta_N]} (L^r)} \le C\|\varphi_N\|_{L^2}\le CN^{\gamma} \label{L2solest}
\end{equation}
holds true.

Now if we find a solution $w$ to the Cauchy problem
\begin{equation}
iw_t+w_{xx}+G(v,w,0)=0,\quad w|_{t=0}=\psi_N, \label{CPP}
\end{equation}
we then obtain a solution $u=v+w$ to the original Cauchy problem (\ref{NLS}) with $N(u)=|u|^{\A-1}u$.  We establish a solution to (\ref{CPP}) in two steps.  We first prove the existence on
the small time interval $[0,\delta_N]$.  Then in the second step, we extend the solution to the time $T_N$.
\\
\textbf{First step}.\quad We can find a fixed point of the operator
\begin{equation*}
\mathscr{T}w\triangleq U(t)\psi_N +i\int^t_0 U(t-s) G(v,w,0) ds
\end{equation*}
in the set 
\begin{equation*}
\{w \in L^{Q_p(r)}_{[0,\delta_N]} (L^r)\,|\, \|w\|_{L^{Q_2(r)}_{[0,\delta_N]} (L^r)} \le RN^{-1} \,\}
\end{equation*}
for a suitable choice of $R>0$.  For example, if $w\in \mathscr{V}$, then
\begin{eqnarray*}
\|\mathscr{T}w \|_{L^{Q_p(r)}_{[0,\delta_N]} (L^r)}
&\le & \|U(t) \psi_N \|_{L^{Q_p(r)}_{[0,\delta_N]} (L^r)} + \left\| \int^t_0 U(t-s) G(v,w,0) ds \right\|_{L^{Q_p(r)}_{[0,\delta_N]} (L^r)}\\
&\le & C_0 N^{-1} + C\delta_N^{1-\frac{\A-1}{4}} \|v\|^{\A-1}_{L^{Q_2(r)}_{[0,\delta_N]} (L^r)} \|w \|_{L^{Q_p(r)}_{[0,\delta_N]} (L^r)}+
C\delta_N^{1-\frac{\A-1}{2p}} \|w\|_{L^{Q_p(r)}_{[0,\delta_N]} (L^r)}^{\A},
\end{eqnarray*}
where we used (\ref{decomp3}) and (\ref{NE1}) in the last step.  By (\ref{deltadef}), (\ref{L2solest}) and the assumption that $w\in \mathscr{V}$, $\delta_N$ and the norms $ \|v\|, \|w\|$ appearing in the right hand side
are controlled using $N$.  A calculation shows that  
\begin{equation*}
\| \mathscr{T}w \|_{L^{Q_2(r)}_{[0,\delta_N]} (L^r)} \le C_0 N^{-1} +CR^{a(\A)}M^{-b({\A)}} N^{-1}
\end{equation*}
for some $a(\A),b(\A)>0$.  This implies that $\mathscr{T}$ maps from $\mathscr{V}$ to itself if we take $M,R$ sufficiently large.  Similarly, we can show that $\mathscr{T}$ is a contraction mapping.  Consequently, we find the solution to (\ref{CPP}) in the space $L^{Q_2(r)}_{[0,\delta_N]} (L^r)$ 
the contraction mapping principle.

\noindent\textbf{Second step}.  The local solution $u:[0,\delta_N] \times \R\to \C$ established above is of the form
\begin{eqnarray*}
u(t)&=&v(t)+w(t)=v(t)+\int^t_0 U(t-s)G(v,w,0)ds +U(t)\psi_N \\
&\triangleq &v(t)+d(t)+l(t).
\end{eqnarray*}
The key is to notice that $d(t)\in L^2$ and its norm $\|d(t)\|_{L^2}$ is far smaller than $\|v(t)\|_{L^2}=\|\varphi_N\|_{L^2}\sim N^{\gamma}$
if $N$ is large.  Hence $\|v(\delta_N)+d(\delta_N)\|_{L^2}\sim N^{\gamma}$ for large $N$.  In addition, by (\ref{decomp3}), we see that
\begin{equation*}
\|U(t-\delta_N)l(\delta_N)\|_{L^{Q_2(r)}_{[0,\delta_N]} (L^r)}
\le C_0 N^{-1}.
\end{equation*}
These imply that $u(\delta_N)$ admits a decomposition $u(\delta_N)=[v(\delta_N)+d(\delta_N)]+l(\delta_N)$ with property similar to (\ref{decomp1})--(\ref{decomp3}).  Therefore,
arguing as in the previous step we construct a solution to the problem
\begin{equation*}
iu_t+u_{xx}+|u|^{\A-1}u=0,\,t>\delta_N,\,x\in \R
\end{equation*}
on the time interval $[\delta_N, 2\delta_N]$.  Repeating similar arguments, we may extend the local solution $u$ to the time $3\delta_N,4\delta_N,\cdots $ as long as
the accumulated Duhamel part from the equation (\ref{CPP}) can be controlled by $N^{\gamma}$.  By (\ref{NE2}) we get an estimate for $\|d(\delta_N)\|_{L^2}$:
\begin{equation*}
\left\| \int^t_0 U(t-s)G(v,w,0)ds \right\|_{L^{\infty}_{[0,\delta_N]} (L^2)} \le CN^{-1+\frac{2(\A-1)}{5-\A}\left(\frac{1}{p}-\frac{1}{2}\right)}.
\end{equation*}
This provides information on how many times one can extend the local solution.  The solution can be extended to the time $k\delta_N$ as long as
\begin{equation}
kN^{-1+\frac{2(\A-1)}{5-\A}\left(\frac{1}{p}-\frac{1}{2}\right)} \le CN^{\gamma}. \label{kdefinition}
\end{equation}
The existence time $T_N=k\delta_N$ appearing in the assertion of Proposition \ref{gexisprop1} is derived by solving (\ref{kdefinition}).



\begin{thebibliography}{99}
\bibitem{BL} J. Bergh and J. L\"ofstr\"om, {\em Interpolation spaces}, Berlin Heidelberg New York: Springer (1976).
\bibitem{Bourgain} J. Bourgain, {\em  Refinements of Strichartz's inequality and
 applications to 2D-NLS with critical
 nonlinearity}. Internat. Math. Res. Notices {\bf} (1998) 253-283.

\bibitem{Bourgainbook}J. Bourgain, {\em New Global well-posedness
	results for nonlinear Schr\"odinger equations}, AMS Publications
	(1999).
\bibitem{Brenner} P. Brenner, V. Thom\'ee and L.B. Wahlbin, {\em Besov Spaces and Applications to Difference Methods for Initial Value Problems}, Lecture Notes in Math. Springer 434.
\bibitem{Carles} R. Carles and L. Mouzaoui, {\em On the Cauchy problem for the Hartree type equation in the Wiener algebra},  Proc. Amer. Math. Soc., {\bf 142} (2014), no.7,2469-2482.
\bibitem{Caz} T. Cazenave, {\em Semilinear  Schr\"odinger equations}, Courant Lect. 
Notes Math. {\bf 10}, New York Univ., Courant Inst. Math. Sci., New York, 2003.
\bibitem{CW} T. Cazenave and F. B. Weissler, {\em
The Cauchy problem for the critical nonlinear Schr\"odinger equation in $H^s$}, 
Nonlinear Anal, {\bf 14} (1990), 807-836.
\bibitem{CVV} T. Cazenave, L. Vega, and M.C.Vilela, {\em A note on the nonlinear Schr\"odinger equation in weak $L^p$ spaces}, Communications in contemporary Mathematics, Vol. 3, No.1 (2001),153-162.
\bibitem{Fefferman} C. Fefferman, {\em Inequalities for strongly singular convolution operators,} Acta math. 124 (1970), 9-36.
\bibitem{GrunrockKdv} A. Gr\"unrock, {\em An improved local well-posedness result for the modified KdV equation}, Int. Math. Res. Not,.{\bf 41} (2004) 3287-3308.
\bibitem{Grunrock} A. Gr\"unrock, {\em Bi- and trilinear Schr\"odinger estimates in one space dimension with applications to cubic NLS and DNLS  },
Int. Math. Res. Not., {\bf 41} (2005), 2525-2558.
\bibitem{Hormander} L. H\"ormander, {\em Estimates for translation invariant operators in $L^p$ spaces,} Acta Math, {\bf 104} (1960), 141-164.
\bibitem{107H}  G. Hoshino and R. Hyakuna, {\em Trilinear $L^p$ estimates with applications to the Cauchy problem for the Hartree-type equation}, preprint.
\bibitem{107H2} R. Hyakuna {\em On the global Cauchy problem for the Hartree equation with rapidly decaying initial data.  }, preprint.
\bibitem{107} R. Hyakuna, {\em On global solutions to nonlinear Schr\"odinger equations with large $L^p$-initial data}, preprint.
\bibitem{107T2} R. Hyakuna and M. Tsutsumi {\em
On the global wellposedness for the nonlinear Schr\"odinger equations with $L^p$- large initial data}, Nonlinear Differential Equations
and Applications NoDEA {\bf 18} (2011), 309-327.
\bibitem{107T} R. Hyakuna and M. Tsutsumi, {\em On existence of global solutions of Schr\"odinger equations with subcritical nonlinearity for $\widehat{L^p}$-initial data}, Proc. Amer. Math. Soc. {\bf 140} (2012), no.11,3905-3920.
\bibitem{Kato} T. Kato, {\em An $L^{q,r}$-theory for nonlinear
	Schr\"odinger equations,} Spectral and scattering theory and
	applications, Adv. Stud. Pure Math., vol. 23, Math. Soc. Japan,
	Tokyo, 1994, 223-238.    
\bibitem{yT} Y. Tsutsumi, {\em $L^2$ -solutions for nonlinear
	Schr\"odinger equations and nonlinear groups}, Funkcial Ekvac., 
        {\bf 30} (1987), 115-125.
\bibitem{VV} A. Vargas and L. Vega, {\em Global wellposedness
	for 1D nonlinear Schr\"odinger equation for data with an
	infinite $L^2$ norm}, J. Math. pures Appl. {\bf 80}, (2001),
	1029-1044.

\bibitem{Zhou} YI.Zhou, {\em  Cauchy problem of nonlinear Schr\"odinger equation with initial data in Sobolev spce $W^{s,p}$ for $p<2$                     }, 
Trans. Amer. Math. Soc., {\bf 362} (2010), 4683-4694.
\end{thebibliography}
\end{document}